\documentclass{article}
\usepackage[T1]{fontenc}
\usepackage{appendix}
\usepackage{calligra}
\usepackage[utf8]{inputenc}
\usepackage[english]{babel}
\usepackage{graphicx,graphics}
\usepackage{rotating}
\usepackage[twoside]{geometry}
\usepackage{epsfig}
\usepackage{indentfirst}
\usepackage[usenames,dvipsnames,svgnames,table]{xcolor}
\usepackage{cmap}
\usepackage{graphicx,graphics}
\usepackage{enumerate}
\usepackage{hyperref}
\hypersetup{
pdftitle={LOG},    
pdfauthor={},    
pdfnewwindow=true,      
colorlinks=true,      
linkcolor=black,         
citecolor=blue}

\usepackage{amssymb,amsfonts,amstext,amsthm}
\usepackage{mathrsfs}
\usepackage[intlimits]{amsmath}
\newcommand{\zerarcounters}
{
\setcounter{equation}{0}
\setcounter{theorem}{0}
}

\newcommand{\OBSI}{\begin{remark}\begin{rm}}
\newcommand{\OBSF}{\end{rm}\end{remark}}
\newcommand{\DEFI}{\begin{definition}\begin{rm}}
\newcommand{\DEFF}{\end{rm}\end{definition}}

\newcommand{\be}{\begin{eqnarray}}
\newcommand{\en}{\end{eqnarray}}
\newcommand{\bee}{\begin{eqnarray*}}
\newcommand{\ene}{\end{eqnarray*}}

\DeclareMathOperator*{\supp}{supp}

\newtheorem{definition}{\bf Definition}[section]
\newtheorem{proposition}{\bf Proposition}[section]

\newtheorem{lemma}{\bf Lemma}[section]
\newtheorem{theorem}{\bf Theorem}[section]

\newtheorem{corollary}{\bf Corollary}[section]

\newtheorem{remark}{Remark}[section]

\newtheorem{example}{Example}[section]

\title{Decay rates given by regularly varying functions for $C_0$-semigroups on Banach spaces}
\author{Genilson Santana and Silas L. Carvalho}

\begin{document}

\date{}
\maketitle

\begin{abstract}
We study rates of decay for (possibly unbounded) $C_0$-semigroups on Banach spaces under the assumption that the norm of the resolvent of the respective semigroup generator grows as a regularly varying function of type $\beta>0$, that is, as $|s|^{\beta}\ell(1+|s|)$ or $|s|^{\beta}/\kappa(1+|s|)$, where $\ell,\kappa$ are arbitrary monotone and slowly varying functions. The main result extends the estimates obtained by Deng, Rozendaal and Veraar (J. Evol. Equ. 24, 99 (2024)) to this setting of regularly varying functions and improves the estimates obtained by Santana and Carvalho (J. Evol. Equ. 24, 28 (2024)) in case  $|s|^{\beta}\log(1+|s|)^b$, with $b\ge 0$.
\end{abstract}

\noindent{\bf{Keywords}}: $C_0$-semigroups, decay given by regularly varying functions, Fourier multipliers, Fourier types, Besov spaces.


\section{Introduction}
\zerarcounters

The asymp
totic theory of semigroups provides tools for investigating the convergence to zero of mild and classical solutions to the abstract Cauchy problem 
\begin{equation}\label{1.0}
\left\{\begin{array}{ll}u'(t)+Au(t)=0, \ \ \  t\geq 0, \\ u(0)=x. \end{array}\right.
\end{equation}
It is known that $\eqref{1.0}$ has a unique mild solution for every $x\in X$ and that the solution depends continuously on $x$ if, and only if, $-A$ generates a $C_0$-semigroup  $(T(t))_{t\geq 0}$ on $X$ (see~\cite{valued,Engel}). In this case,
the unique solution $u$ to $\eqref{1.0}$ is given by $u(t)=T(t)x, \  \ \forall~t\geq 0$, and if $x\in \mathcal{D}(A)$, then
$u \in C^{1}([0,\infty),X)$ (see \cite[Proposition~II.6.2]{Engel}).

In the last two decades, there have been published a great number of works in semigroup theory devoted to the study of decay rates for classical solutions (assuming some regularity for the initial data) and some spectral properties for the infinitesimal generator  (see~\cite{Duykaerts, chill, chill1, rozendaal,Stahn}). Such results arise from concrete problems, for example, in the study
 of the damped wave equation \cite{nalini, burq, chill2, lebeau}.

However, in many concrete problems, the associated semigroup $(T(t))_{t\ge 0}$ is not bounded, or at least one cannot prove such boundedness (see e.g,~\cite{Paunonen, renardy, roz2}). Hence, the problem of proving polynomial decay for not necessarily bounded semigroups is of great importance. In this vein, we highlight the pioneering work of Bátkai, Engel, Pr\"uss and Schnaubelt~\cite{batki}. 

Given the relevance of the result concerning the optimal decay of bounded semigroups whose resolvent exhibits polynomial growth (see \cite{tomilov}), the problem of obtaining more refined rates, now under the assumption that the resolvent grows according to functions which are finer than polynomial ones, has become of great interest (see \cite{chill}). In this context, by  addressing the same question without assuming the boundedness of the semigroup, we seek to establish new decay rates, considering that the resolvent growth is described by the product of a polynomial function and a slowly varying function (for example, $f(\lambda)=\lambda^{\beta} \log(1+\lambda)^{b}$, with $\beta,b>0$; see Subsection~\ref{sub222} for discussion involving regularly varying functions).

In 2018, Rozendaal and Veraar~\cite{rozendaal} obtained new estimates that improved the estimates obtained so far for the decay rates of unbounded semigroups. These new estimates, presented in the next theorem, take into account geometric properties of the underlying Banach space (namely, its Fourier type). 
\begin{theorem}[\texorpdfstring{\cite[Theorem~4.9]{rozendaal}}{rozendaal}]\label{theoem 1.9}
\begin{rm}
 Let $(T(t))_{t\geq 0}$ be a $C_0$-semigroup with generator $-A$ defined in a Banach space
$X$ with Fourier type $p\in [1,2]$, and let $\frac{1}{r}=\frac{1}{p}-\frac{1}{p'}$ (where $\frac{1}{p}+\frac{1}{p'}=1$). Suppose that  $\overline{\mathbb{C}_{-}}\subset \rho(A)$ and that there exist $\beta,C \geq 0$ such that  $\|(\lambda+A)^{-1}\|_{\mathcal{L}(X)}\leq C
(1+|\lambda|)^{\beta}$ for each $\lambda \in \overline{\mathbb{C}_{-}} $. Let $\tau>\beta+1/r$; then, for each $\varepsilon>0$, there exists $C_\varepsilon \geq 0$ such that for each $t \ge 1$,
\begin{equation}\label{eeq3}
    \|T(t)(1+A)^{-\tau}\|_{\mathcal{L}(X)}\leq C_{\varepsilon}t^{1- \frac{\tau-1/r}{\beta}+\varepsilon}.
\end{equation}
\end{rm}    
\end{theorem}

In 2024, Santana and Carvalho~\cite{nos} went a step further by considering that $\|(\lambda+A)^{-1}\|_{\mathcal{L}(X)}\lesssim (1+|\lambda|)^{\beta}\log(2+|\lambda|)^b$, with $b\ge 0$ (a particular
example of a regularly varying function of index $\beta$; see Definition~\ref{reg}). The result was obtained by combining the techniques developed in \cite{rozendaal} for Fourier multipliers with the techniques developed by Batty, Chill and Tomilov in \cite{chill} for the functional calculus of sectorial operators. 
\begin{theorem}[\texorpdfstring{\cite[Theorem~1.13]{nos}}{nos}]\label{theo4.51}
\begin{rm}
Let $\beta>0$, $b\geq 0$ and let $(T(t))_{t\geq 0}$ be a $C_0$-semigroup with generator $-A$  defined in a Banach space $X$ with Fourier type $p\in [1,2]$. Suppose that $\overline{\mathbb{C}_{-}}\subset \rho(A)$ and that for each $\lambda \in \mathbb{C}$ with $\text{Re}(\lambda)\le 0$,
\[\|(\lambda+A)^{-1}\|_{\mathcal{L}(X)}\lesssim (1+|\lambda|)^{\beta}(\log(2+|\lambda|))^b.\]
 Let $r\in [1,\infty]$ be such that $\frac{1}{r}=\frac{1}{p}-\frac{1}{p'}$ and let $\tau > 0$ be such that $\tau>\beta+\frac{1}{r}$. Then, for each $\delta>0$, there exist constants $c_{\delta,\tau}\ge 0$ and $t_0>1$ such that for each $t\geq t_0$,
\begin{equation}\label{eqTh1.1}
 \|T(t)(1+A)^{-\tau}\|_{\mathcal{L}(X)}\leq c_{\delta,\tau}t^{1-\frac{\tau-r^{-1}}{\beta}}\log(t)^{\frac{b(\tau-r^{-1})}{\beta}+\frac{1+\delta}{r}}.
\end{equation}
\end{rm}
\end{theorem}

\begin{remark} \begin{rm} In particular, if one let $b=0$ in Theorem~\ref{theo4.5}, then for each $t\geq t_0>1$,
\begin{equation}\label{eeq3a}
    \|T(t)(1+A)^{-\tau}\|_{\mathcal{L}(X)}\lesssim t^{1-\frac{\tau-r^{-1}}{\beta}}\log(t)^{\frac{1+\delta}{r}},
\end{equation}
a result that improves the estimate presented in~\eqref{eeq3}.
\end{rm}
  \end{remark}

Recently, Deng, Rozendaal and Veraar~\cite{deng}, by using the boundedness of Fourier multipliers defined in some specific Besov spaces, improved the estimate presented in~\eqref{eeq3a} by removing the logarithm factor $\log(t)^{\frac{1+\delta}{r}}$. 

\begin{theorem}[\texorpdfstring{\cite[Theorem~1.1]{deng}}{deng}]\label{theo13}
\begin{rm}
Let $\beta>0$ and let $(T(t))_{t\geq 0}$ be a $C_0$-semigroup with generator $-A$ defined in a Banach space $X$ with Fourier type $p\in [1,2]$. Suppose that $\overline{\mathbb{C}_{-}}\subset \rho(A)$ and that for each $\lambda \in \mathbb{C}$ with $\text{Re}(\lambda)\le 0$,
\[\|(\lambda+A)^{-1}\|_{\mathcal{L}(X)}\lesssim (1+|\lambda|)^{\beta}.\]
 Let $r\in [1,\infty]$ be such that $\frac{1}{r}=\frac{1}{p}-\frac{1}{p'}$, and let $\tau > 0$ be such that $\tau>\beta+\frac{1}{r}$. Then, there exists a constant $c_{\tau}\ge 1$ such that for each $t\geq 1$,
\begin{equation*}
 \|T(t)(1+A)^{-\tau}\|_{\mathcal{L}(X)}\leq c_{\tau}t^{1-\frac{\tau-r^{-1}}{\beta}}.
\end{equation*}
\end{rm}
\end{theorem}
By considering this new result, one may ask if the logarithmic factor $\log(t)^{\frac{1+\delta}{r}}$ can also be dropped from~\eqref{eqTh1.1}, that is, if it can also be dropped in case $b>0$. Our main goal in this work is to address this question not only for the result presented in Theorem~\ref{theo4.51}, but also for the case where the growth of the norm of the resolvent is given by a more general regularly varying function (as originally discussed in~\cite{chill} for bounded $C_0$-semigroups). One should compare the estimates~\eqref{eq16} (in case $\ell(t)=\log(2+t)$, with $t\ge 1$) and~\eqref{eqTh1.1}.
\begin{theorem}\label{theo4.5}
\begin{rm}
Let $\beta>0$, $b\geq 0$ and let $(T(t))_{t\geq 0}$ be a $C_0$-semigroup with generator $-A$  defined in a Banach space $X$ with Fourier type $p\in [1,2]$. Suppose that $\overline{\mathbb{C}_{-}}\subset \rho(A)$ and let $r\geq 0$ be such that $\frac{1}{r}=\frac{1}{p}-\frac{1}{p'}$. 
\begin{enumerate}[(i)]
\item Assume that for each $\lambda \in \mathbb{C}$ with $\text{Re}(\lambda)\le 0$,
\[\|(\lambda+A)^{-1}\|_{\mathcal{L}(X)}\lesssim (1+|\lambda|)^{\beta}\ell(1+|\lambda|),\]
where $\ell$ is an increasing and slowly varying function. Then, for each $\tau > 0$ such that $\tau>\beta+\frac{1}{r}$, there exist $c_\tau>0$ and $t_0\geq 1$ so that for each $t\geq t_0$,
 \begin{equation}\label{eq16}
 \|T(t)(1+A)^{-\tau}\|_{\mathcal{L}(X)}\leq c_{\tau}t^{1-\frac{\tau-r^{-1}}{\beta}}(\ell(t^{1/\beta}))^{\frac{\tau-r^{-1}}{\beta}}.
\end{equation}
\item Assume that for each $\lambda \in \mathbb{C}$ with $\text{Re}(\lambda)\le 0$,
\[\|(\lambda+A)^{-1}\|_{\mathcal{L}(X)}\lesssim (1+|\lambda|)^{\beta}(1/\kappa(1+|\lambda|)),\]
where $\kappa$ is an increasing and slowly varying function. Then, for each $\tau > 0$ such that $\tau>\beta+\frac{1}{r}$, there exist $\tilde{c}_\tau>0$ and $t_0\geq 1$ so that for each $t\geq t_0$,
 \begin{equation*}
 \|T(t)(1+A)^{-\tau}\|_{\mathcal{L}(X)}\leq \tilde{c}_{\tau}t\left(\frac{1}{t\kappa(t^{1/\beta})}\right)^{\frac{\tau-r^{-1}}{\beta}}.
\end{equation*}
\end{enumerate}
\end{rm}
\end{theorem}

The next result is just Theorem~\ref{theo4.5} in case  $X$ is a Hilbert space.

\begin{corollary}\label{theo4.6}
 \begin{rm}
Let $\beta$, $\ell$, $\kappa$, $A$  and $(T(t))_{t\geq 0}$ be as in the statement of Theorem~\ref{theo4.5}, and let $X$ be a Hilbert space. Let $\tau>\beta$. Then,    there exist constants $c_{\tau},\tilde{c}_{\tau}\geq 0$ and $t_0\ge 1$ such that for each $t\geq t_0$,
\begin{enumerate}[(i)]
    \item \qquad\qquad\qquad\qquad\qquad$\|T(t)(1+A)^{-\tau}\|_{\mathcal{L}(X)}\leq c_{\tau}t^{1-\frac{\tau}{\beta}}\ell(t^{1/\beta})^{\frac{\tau}{\beta}}$.
\item 
 \qquad\qquad\qquad\qquad\qquad $\|T(t)(1+A)^{-\tau}\|_{\mathcal{L}(X)}\leq \tilde{c}_{\tau}t\left(\dfrac{1}{t\kappa(t^{1/\beta})}\right)^{\frac{\tau}{\beta}}$.
\end{enumerate}
\end{rm}   
\end{corollary}

The text is organized as follows: in Section~\ref{SPre} we present some definitions and preliminary results needed in the proof of Theorem~\ref{theo4.5}, which is presented in details in Section~\ref{three}.

\section{Preliminaries}\label{SPre}
\zerarcounters
In this section we set the notation used throughout the text and present some definitions and auxiliary results needed in the proof of Theorem~\ref{theo4.5}.

\subsection{Notation}

We set $\mathbb{N}_0:=\{0,1,2,3,\ldots\}$, $\mathbb{R}_{+}:=\{x\in \mathbb{R}\mid x>0\}$, $\mathbb{C}_{\pm}:=\{z\in \mathbb{C}\mid \text{Re}(\lambda)\gtrless 0\}$ and for each $\omega\in (0,\pi]$, set $S_{\omega}:= \{z\in \mathbb{C} \setminus\{0\}\mid |\text{arg}(z)|<\omega\}$ and $S_0:=(0,\infty)$. The H\"older conjugate of  $p\in [0,\infty]$ is denoted by $p'$, so that $\frac{1}{p}+\frac{1}{p'}=1$.

Let $\Omega$ be an open connected subset of $\mathbb{C}$; we denote by $H^{\infty}_0(\Omega)$ the set of holomorphic functions $f:\Omega\rightarrow \mathbb{C}$ for which there exist constants $C>0$ and $s>0$ such that for each $z\in\Omega$, $|f(z)|\leq C \min\{|z|^s,|z|^{-s}\}$. 

We denote by $\mathcal{L}(X,Y)$ the (Banach) space of bounded linear operators from $X$ to $Y$ (both $X$ and $Y$ are non-trivial Banach spaces over $\mathbb{C}$), with $\mathcal{L}(X):=\mathcal{L}(X,X)$. 

For a linear operator $A$ defined in $X$, we denote by $D(A)$, $\sigma(A)$ and $\rho(A)$ the domain, the spectrum and the resolvent set of $A$, respectively. We denote by $R(\lambda, A)=(\lambda-A)^{-1}$ 
the resolvent operator of $A$ at $\lambda \in \rho (A)$.

We denote by $\mathcal{S}(\mathbb{R};X)$ and $\mathcal{S}'(\mathbb{R};X)$ the spaces of $X$-valued Schwarz functions and tempered distributions, respectively. 

We say that the Banach space $X$ has \textit{Fourier type} $p\in [1,2]$ if the Fourier transform $\mathcal{F}:L^p(\mathbb{R}; X)\rightarrow L^{p'}(\mathbb{R}; X)$ is bounded; we then set $\mathcal{F}_{p,X}:= \|\mathcal{F}\|_{\mathcal{L}(L^p(\mathbb{R}; X),L^{p'}(\mathbb{R}; X))}$. A Banach space $X$ has \textit{Fourier cotype} $q\in [2, \infty]$ if $X$ has Fourier type $q^{\prime}$. Each Banach space has Fourier type $1$, and $X$ has Fourier type $2$ if, and only if, $X$ is isomorphic to a Hilbert space (see \cite[Theorem 2.1.18]{tuomas}).

We write $|f(z)|\lesssim |g(z)|$ to indicate that there exist $C>0$ and $0\neq z_0 \in \mathbb{C}$ such that for each $|z|\ge |z_0|$, $|f(z)|\leq C |g(z)|$. 
\subsection{Some important classes of functions}

\subsubsection{Complete Bernstein functions}

In this subsection, we recall the definitions and some properties of the complete Bernstein and regularly varying functions of type $\beta$. We refer to~\cite{ber} for details (see also~\cite{chill}).
\begin{definition}
\rm{A function $f\in C^{\infty}(0,\infty)$ is a Bernstein function if there exist
constants $a,b\geq 0$ and a positive Radon measure $\mu$, 
defined over the Borel subsets of $(0,\infty)$, such that for each $\lambda>0$,
\begin{eqnarray*}
f(\lambda)= a+b\lambda+ \int_{0^{+}}^{\infty}(1-e^{-\lambda s}) \text{d}\mu(s), 
\end{eqnarray*}
with
\begin{eqnarray*}
\int_{0^{+}}^{\infty}\frac{s}{s+1}\text{d}\mu(s)<\infty.
\end{eqnarray*}
}
\end{definition}

The triple $(a,b,\mu)$ determines $f$ uniquely and vice-versa (see \cite[Theorem 3.2]{ber}), and it is called the Lévy-Khintchine triple of $f$. Every Bernstein function can also be extended to a holomorphic function in $\mathbb{C}_+$ (this is \cite[Proposition 3.6]{ber}).

Now we consider a subclass of the Bernstein functions, the so-called complete Bernstein functions.

\begin{definition}[\texorpdfstring{\cite[Definition~6.1]{ber}}{ber}]\label{DCBF}
 {\rm A function $f\in C^{\infty}(0,\infty)$ is called
a {\bf{complete Bernstein function}} if it is a Bernstein function and if the measure
$\mu_{LK}$ in the Lévy-Khintchine triple has a completely monotone density with
respect to the Lebesgue measure. The set of all complete Bernstein functions is denoted by $\mathcal{CBF}$.}
\end{definition}

It follows from \cite[Theorem 6.2-(vi)]{ber} that each $f\in\mathcal{CBF}$ admits a representation of the form
\begin{equation}\label{eqq5}
    f(\lambda)=a+b\lambda+ \int_{0^{+}}^{\infty} \frac{\lambda}{\lambda+s}\text{d}\mu(s), \ \ \lambda>0,
\end{equation}
with $a,b\geq 0$ constants and $\mu$ a positive Radon measure defined over the Borel subsets of $(0, \infty)$ that satisfies
\begin{equation*}
  \int_{0^{+}}^{\infty}  \frac{1}{s+1}\text{d}\mu(s)<\infty.
\end{equation*}

The representation $\eqref{eqq5}$  is  unique (that is, the triple $(a,b,\mu)$ is unique) and it is called the Stieltjes
representation for $f$ (see \cite[Chapter 6]{ber} for details). Note that
\begin{eqnarray*}
    a=\lim_{\lambda \rightarrow 0^{+}} f(\lambda) \ \ \text{and} \ \ b=\lim_{\lambda \rightarrow \infty} \frac{f(\lambda)}{\lambda}.
\end{eqnarray*}
It is important to note that the complete Bernstein functions can be extended
holomorphically to $\mathbb{C}\setminus (-\infty,0]$ (see \cite[Theorem~6.2-(v)]{ber}).

The following result will be used in the proof of Proposition~\ref{prop3.1}.
\begin{proposition}[\texorpdfstring{\cite[Proposition~2.4]{chill}}{chill}]\label{prop21}
\rm{Let $f$ be a complete Bernstein function with Stieltjes representation $(0, 0,\mu)$. Let $\lambda \in S_\pi$ and $\varphi=\text{arg}(\lambda)$. Then,
\begin{equation*}
    \cos(\varphi/2)f(|\lambda|)\le |f(\lambda)|\le \sec(\varphi/2) f(|\lambda|).
\end{equation*}

}
 
\end{proposition}
\subsubsection{Regularly varying functions}\label{sub222}

\begin{definition}
\begin{rm}
 Let $a\in\mathbb{R}$ and let $\ell:[a,\infty)\rightarrow \mathbb{R}$ be a strictly positive measurable function 
 such that for each $\lambda>0$,
\begin{equation*}
    \lim_{s\rightarrow \infty} \frac{\ell(\lambda s)}{\ell(s)}=1.
\end{equation*}
Then, $\ell$ is said to be \textbf{slowly varying}. 
\end{rm}
\end{definition} 

\begin{definition}\label{reg}
\rm{Let $\alpha \in \mathbb{R}$ and let $f:\mathbb{R}_+\rightarrow \mathbb{R}_+$ be a positive measurable function. Then, $f$ is called \textbf{regularly varying of index $\alpha$} if, for each $\lambda\ge 1$,
\[\lim_{s\to \infty} \frac{f(\lambda s)}{f(s)}=\lambda^\alpha.\]}
\end{definition}

If $\alpha \in \mathbb{R}$ and if $f:\mathbb{R}_+\rightarrow \mathbb{R}_+$ is a regularly varying function of index $\alpha$, then there exist a slowly varying function $\ell:\mathbb{R}_+\rightarrow\mathbb{R}_+$ and $s_0> 0$ such that for each $s\ge s_0$, $f(s)=s^{\alpha}\ell(s)$ (see~\cite{reg}~for more details).
\begin{definition}[\texorpdfstring{\cite[Example~2.14]{chill}}{chill}] \label{sg}
\rm{Let $\ell$ be a slowly varying function on $\mathbb{R}_{+}$ and assume that $g(s)= s^{\beta} \ell(s)$ is increasing. One defines the so-called associated Stieltjes function $S_g:\mathbb{R}_+\rightarrow\mathbb{R}_+$ by the law
\begin{equation*}
    S_{g}(\lambda):=\int_{0^{+}}^{\infty} \frac{s^{\beta}\ell(s)}{(s+\lambda)^2} \text{d}s  
    \end{equation*}
if the integral is finite (see in~\cite{chill} for more details).}
\end{definition}

\begin{example}\label{ex22}
\rm{Let $\ell$ be a slowly varying function on $\mathbb{R}_{+}$ and let $S_g$ be the function presented in Definition~\ref{sg}.  
\begin{enumerate}[(a)]
\item  The function $f:\mathbb{R}_+\rightarrow \mathbb{R}_+$, given by the law $f(\lambda)=S_g\left(\dfrac{1}{\lambda}\right)$, is a complete Bernstein function.
\item It follows from \cite[Proposition 2.4-(a)]{nos} (or \cite[Theorem~2.2]{chill}) that the function $\Psi:\mathbb{R}_+\rightarrow \mathbb{R}_+$, given by the law $\Psi(\lambda)= \dfrac{\lambda}{S_g\left(1/\lambda\right)}$, is a complete Bernstein function. Moreover, it follows from \cite[Proposition~2.4-(c)]{nos} that for each $\varepsilon \in (0,1)$, $\Psi_\varepsilon(\lambda):=\Psi(\lambda^\varepsilon)$ is a complete Bernstein function.
\end{enumerate}
}
\end{example}

The next results will be used in the proofs of Proposition~\ref{prop3.1} and Theorem~\ref{theo4.5}.
\begin{proposition}[\texorpdfstring{\cite [Corollary~2.6-(a)]{chill}}{chill}]\label{cor2}
\begin{rm}
Let $\ell$ be a slowly varying function and let $\gamma>0$. 
\begin{enumerate}[(a)]
    \item Then, there exist positive constants $C, c$ such that for each sufficiently large $s, t$,  with $t\geq s$, 
\begin{equation*}
   c\left(\frac{s}{t}\right)^{\gamma}\leq \frac{\ell(t)}{\ell(s)}\leq C\left(\frac{t}{s}\right)^{\gamma}.
\end{equation*}
\item 
\[\lim_{s\to \infty} s^{\gamma}\ell(s)=\infty \ \ \ \ \text{and} \ \ \ \lim_{s\to \infty} s^{-\gamma}\ell(s)=0\]
\end{enumerate}
\end{rm}
\end{proposition}
\begin{theorem}[Karamata, \texorpdfstring{\cite[Theorem~2.15]{chill}}{chill}] \label{lema21}
\rm{Let $g:\mathbb{R}_+\rightarrow \mathbb{R}_+$ be an increasing function and let $S_g$ be the  associated Stieltjes function (see Definition~\ref{sg}). Let $0<\sigma \le 1$ and let $\ell:\mathbb{R}_+\rightarrow \mathbb{R}_+$ be a slowly varying function. Then, the following statements are equivalent.
 \begin{enumerate}[(i)]
     \item $\displaystyle
{\lim_{s\to \infty} \dfrac{g(s)}{s^{1-\sigma}\ell(s)}}=1$.
     \item $ \displaystyle
{\lim_{\lambda\to \infty} \dfrac{S_g(\lambda)}{\Gamma(\sigma)\Gamma(2-\sigma)\lambda^{-\sigma}\ell (\lambda)}=1} $, where $\Gamma(\sigma)$ stands for the Gamma function at $\sigma>0$.
 \end{enumerate}}  
\end{theorem}

\subsection{Fourier Multipliers and Besov spaces}

\subsubsection{Fourier Multipliers}
Let $X$ and $Y$ be Banach spaces and let $m: \mathbb{R} \rightarrow \mathcal{L}(X,Y)$ be a $X$-strongly measurable map (i.e. the map $\xi\mapsto m(\xi)x$ is a strongly measurable $Y$-valued map for every $x\in X$). One says that $m$ is of {\it{moderate growth at infinity}} if there exist $\beta \geq 0$ and $g\in L^{1}(\mathbb{R})$ such that for each~$\xi\in \mathbb{R}$,
\begin{equation*}
    \frac{1}{(1+|\xi|)^{\beta}}\|m(\xi)\|_{\mathcal{L}(X,Y)}\lesssim g(\xi).
\end{equation*}

For such a measurable $m$, one defines the {\bf{\it{Fourier multiplier operator}}} associated with $m$, $T_m:\mathcal{S}(\mathbb{R};X)\rightarrow \mathcal{S}'(\mathbb{R};Y)$, by the law
\begin{equation*}
   T_m(f):= \mathcal{F}^{-1}(m \cdot \mathcal{F}{f}), \qquad \qquad \forall~f\in \mathcal{S}(\mathbb{R};X);
\end{equation*}
$m$ is called the {\it{symbol}} of $T_m$. For $p\in [1,\infty)$ and $q\in [1, \infty]$, let $\mathcal{M}_{p,q}(\mathbb{R};\mathcal{L}(X,Y))$ denote the set of all $X$-strongly measurable maps $m: \mathbb{R}\rightarrow \mathcal{L}(X,Y)$ of moderate growth such that $T_m\in \mathcal{L}(L^{p}(\mathbb{R};X),L^{q}(\mathbb{R};Y))$,  and set $\|m\|_{\mathcal{M}_{p,q}(\mathbb{R};\mathcal{L}(X,Y))}:=\|T_m\|_{\mathcal{L}(L^{p}(\mathbb{R};X),L^{q}(\mathbb{R};Y))}$.
\subsubsection{Besov spaces}
We suggest~\cite{tuomas, roz} for a more detailed account of Besov spaces.  Let $\psi \in \mathcal{S}(\mathbb{R})$ be such that $\supp \mathcal{F}(\psi) \subset\left[\frac{1}{2},2\right]$, $\mathcal{F}(\psi)\geq0$, and for each $\xi>0$, $\displaystyle{\sum_{k=-\infty}^{\infty}\mathcal{F}(\psi)(2^{-k}\xi)=1}$. 
 
Now, let $(\varphi_k)_{k\geq 0} \subset \mathcal{S}(\mathbb{R})$ (the so-called inhomogeneous Littlewood–Paley  sequence) be such that for each $\xi \in \mathbb{R}$,
\begin{equation*}
\mathcal{F}(\varphi_k)(\xi)=\mathcal{F}(\psi)(2^{-k}|\xi|)\ \ \text{for} \ \ k\geq 1 \ \ \text{and} \ \   \mathcal{F}(\varphi_0)(\xi)=1- \sum_{k=1}^{\infty}\mathcal{F}(\varphi_k)(\xi).
\end{equation*}

Let $(\varphi_k)_{k\geq 0} \subset \mathcal{S}(\mathbb{R})$ be defined as above, let $X$ be a Banach space and let $s\in \mathbb{R}$, $p,q\in [1,\infty]$. The \textbf{inhomogeneous Besov space}  $B^{s}_{p,q}(\mathbb{R}; X)$ is the normed space of all $f\in S'(\mathbb{R};X)$ such that $\varphi_k*f\in L^{p}(\mathbb{R},X)$ for each $k\in \mathbb{N}_0$, 
and such that the norm
 \begin{equation*}
   \|f\|_{B^{s}_{p,q}(\mathbb{R}; X)}:=\left\|\left(2^{sk}\|\varphi_k*f\|_{L^{p}(\mathbb{R};X)}\right)_{k\geq 0}\right\|_{\ell^q}
 \end{equation*}
 is finite.
 
 Actually, $(B^{s}_{p,q}(\mathbb{R}; X),\|\cdot\|_{B^{s}_{p,q}(\mathbb{R}; X)})$ is a Banach space, and the continuous inclusions
  \begin{equation*}
    \mathcal{S}(\mathbb{R};X)\subset B^{s}_{p,q}(\mathbb{R}; X) \subset \mathcal{S}'(\mathbb{R};X)
  \end{equation*}
 hold (see \cite[Proposition~14.4.3]{tuomas}). Moreover, it follows from \cite[Proposition~14.4.18]{tuomas} that for each $1\le p,q\leq \infty$ and each $r,s \in \mathbb{R}$ with $r>s$, 
 \begin{equation}
B^{r}_{p,q}(\mathbb{R}, X)\subset B^{s}_{p,q}(\mathbb{R}, X) \ \ \text{and} \ \ B^{0}_{p,1}(\mathbb{R}, X)\subset L^p( \mathbb{R}, X) \subset B^{0}_{p,\infty}(\mathbb{R}, X).
 \end{equation}

 The next result presents a sufficient condition for $T_m$ to be a bounded operator when defined in a Besov space. This result justifies the use of such spaces in~\cite{deng} and in this work. 
\begin{proposition}[\texorpdfstring{\cite[Proposition~2.2]{deng}}{deng}]\label{prop22}
\rm{Let $X$ and $Y$ be Banach spaces with Fourier type $p\in [1,2]$ and let $m:\mathbb{R}\rightarrow \mathcal{L}(Y,X)$ be an $X$-strongly measurable function, with 
\[\sup_{\xi \in \mathbb{R}} \|m(\xi)\|_{\mathcal{L}(Y,X)}<\infty.\] 
Then, 
\[T_m: \mathcal{B}^{\frac{1}{p}-\frac{1}{p'}}_{p,p}(\mathbb{R},Y)\rightarrow L^{p'}(\mathbb{R},X)\]
is bounded.}
\end{proposition}

\subsection{Sectorial operators and Interpolation spaces}

\subsubsection{Sectorial operators}

For the required background on sectoral operators, we refer
to \cite{haase}. 

For each $\omega\in (0,\pi]$, set $S_{\omega}:= \{z\in \mathbb{C} \setminus\{0\}\mid |\text{arg}(z)|<\omega\}$; set also $S_0:=(0,\infty)$.

\begin{definition}
  \rm{A linear operator $A:D(A)\subset X\rightarrow X$ is called {\bf{sectorial}} of angle $\omega$ if $\sigma(A)\subset \overline{S_{\omega}}$ and if $M(A, \omega):=\sup\{\|\lambda R(\lambda,A)\|_{\mathcal{L}(X)} \mid \lambda \in \mathbb{C}\setminus \overline{S_{\omega'}}, \  \omega'\in (\omega, \pi)\}<\infty$. One denotes the set of such operators by $\text{Sect}_X(\omega)$.}
\end{definition}  

Set $\omega_A:=\min\{\omega\in (0,\pi)\mid A\in \text{Sect}_X(\omega)\}$, which is the minimal angle for which $A$ is sectorial.  
\begin{remark}
\begin{rm}
Let $A:D(A)\subset X\rightarrow X$ be a linear operator for which $(-\infty,0) \subset \rho(A)$ and
\begin{equation*}
    M_A:= M(A,\pi)=\sup_{t>0}t\|(t+A)^{-1}\|_{\mathcal{L}(X)}<\infty;
\end{equation*}
then, it follows that $A\in \text{Sect}_X(\pi-\arcsin{\left(1/M_A\right))}$. 
    \end{rm}
\end{remark}

Now we recall some basic properties of the functional calculus of sectorial
operators based on complete Bernstein functions. We use \cite{chill} as a reference for our discussion (see also~\cite{gomilko, batty1, batty33, tomilov1}).


\begin{definition}[\texorpdfstring{\cite[Definition 3.3]{chill}}{chill}]\label{Dchill}
\begin{rm}
Let $A\in\text{Sect}_X(\omega_A)$ be densely defined and let $f\in\mathcal{CBF}$, with Stieltjes representation $(a,b,\mu)$. One defines the linear operator $f_0(A):D(A)\rightarrow X$  by the law
\begin{equation}\label{cbf}
f_0(A)x=ax+bAx+\int_{0+}^{\infty} A(A+\lambda)^{-1}x \text{d}\mu(\lambda), \ \ \ x\in D(A).
\end{equation}
Set $f(A):= \overline{f_0(A)}$. We call the linear operator $f(A)$ a complete Bernstein function of $A$. 
\end{rm}
\end{definition}

\begin{theorem}[\texorpdfstring{\cite[Theorem~3.6]{chill}}{chill}]\label{The2.3}
\begin{rm}
 Let $A$ be a sectorial operator in a Banach space $X$ and let
$f\in\mathcal{CBF}$. Then, $f(A)$ is sectorial.
\end{rm}
\end{theorem}

\begin{proposition}[Relation~(3.12) in~\cite{chill}]\label{fsec}
    \rm{ Let $A$ be an invertible sectorial operator in a Banach space $X$ and let $f\in\mathcal{CBF}$. Then, for each $x\in X\setminus\{0\}$,
    \begin{equation*}
        \|f(A^{-1})x\|_X\le \|x\|f\left(\frac{\|A^{-1}x\|_X}{\|x\|_X}\right).
    \end{equation*}
    }
\end{proposition}
\subsubsection{Interpolation spaces}

For more details on real interpolation Banach spaces we suggest~\cite{haase, tou, Ro}. An interpolation couple is a pair $(X,Y)$ of Banach spaces which are continuously embedded in a Hausdorff topological vector space $Z$. The real interpolation space of $(X,Y)$ with parameters $\theta \in [0,1]$ and $q\in [1,\infty]$ is denoted by $(X,Y)_{\theta,q}$.  

 Let $A \in \text{Sect}_X(\omega)$, $\tau \in (0,\infty)$ and $1\leq q \leq \infty$. Then, the real interpolation space associated with $A$, $\tau$ and $q$ is given by
 \begin{equation*}
      D_A(\tau,q):= (X,D(A^{\alpha}))_{\tau/\alpha,q},
 \end{equation*}
  where $\alpha \in (\tau,\infty)$ is arbitrary (that is, this definition does not depend on the choice of $\alpha>\tau$). It follows from \cite[Corollary 6.6.3]{haase} that for each $\tau>0$,
  \begin{equation}\label{inter}
      D_A(\tau,1)\subset D(A^{\tau})\subset D_A(\tau, \infty).
  \end{equation}
Moreover, $D(A^\alpha)$ is a dense subset of $D_A(\tau,q)$ for each $\alpha >\tau $ and each $q<\infty$ (see \cite[Theorem 6.6.1]{haase}).

The following theorem plays an important role in the proofs of some results presented in the next section.
\begin{theorem}[\texorpdfstring{\cite[Theorem C.3.3]{tou}}{tou}]\label{inter1}
\rm{Let $(X_0,X_1)$ and
 $(Y_0,Y_1)$ be interpolation couples. Suppose that $T: X_0+X_1\rightarrow Y_0+Y_1$ is a linear  operator which maps $X_0$ into $Y_0$ and $X_1$ into $Y_1$ with norms $A_0:=\|T\|_{\mathcal{L}(X_0,Y_0)}$ and $A_1:=\|T\|_{\mathcal{L}(X_1,Y_1)}$. Then, for each $\theta \in (0,1)$ and each $p\in [1,\infty]$, the operator $T$ maps $(X_0,X_1)_{\theta,p}$ into  $(Y_0,Y_1)_{\theta,p}$,  and one has
 \begin{equation*}
     \|T\|_{\mathcal{L}((X_0,X_1)_{\theta,p},(Y_0,Y_1)_{\theta,p})} \le A^{1-\theta}_0A^{\theta}_1.
 \end{equation*}
 }
    
\end{theorem}

\section{Main Result}\label{three}
\zerarcounters

We let $\beta>0$ throughout this section, whereas $\ell,\kappa:\mathbb{R}_+\rightarrow\mathbb{R}_+$ stand for increasing slowly varying functions. We define, for each $\varepsilon \in (0,1)$, $\Psi_{\varepsilon}:\mathbb{R}_+\rightarrow \mathbb{R}_{+}$ by the law $\Psi_{\varepsilon}(s):=\dfrac{s^{\varepsilon}}{S_{h_{\varepsilon}}(1/s^\varepsilon)}$, where $S_{h_{\varepsilon}}$ stands for the associated Stieltjes function to $h_{\varepsilon}(s):= \ell (s^{1/\varepsilon})$ (see Definition~\ref{sg}); we also define $\Phi: \mathbb{R_{+}}\rightarrow \mathbb{R}_{+}$ by the law $\Phi(t)=t S_{\kappa}(t)$. Finally, for each $A\in \text{Sect}_X(\omega_A)$, set $B:= 1+A$.

In what follows, we use Proposition~\ref{prop3.1} (which is a version of \cite[Proposition~3.3]{nos} and \cite[Proposition~3.4]{rozendaal} for our setting) to prove Proposition~\ref{Prop3.1}. Then, we combine Proposition~\ref{Prop3.1} with Theorem~\ref{teo32} to prove Theorem~\ref{teo31}. Finally, Theorem~\ref{theo4.5} follows, among other results, from Theorem~\ref{teo31}. 

\begin{proposition}\label{prop3.1}
\begin{rm}
Let $A\in \text{Sect}_X(\omega_A)$ be such that $\overline{\mathbb{C}_{-}}\subset \rho(A)$, and let $a\geq 0$. 
\begin{enumerate}[(1)]
\item If for each $\xi\in \mathbb{R}$, 
\begin{equation*}\label{eeq23}
    \|(i\xi+A)^{-1}\|_{\mathcal{L}(X)}\lesssim (1+|\xi|)^{\beta} \ell(1+|\xi|),
\end{equation*}
then for each $\varepsilon>0$, the family
\begin{equation}\label{eeq24}
\{\ell(|\lambda|)^{a}\|(\lambda+A)^{-1}(1+A)^{-\beta}\Psi_{\varepsilon}(B^{-1})^{1+a}\|_{\mathcal{L}(X)}\mid \lambda \in i\mathbb{R}, |\lambda|\geq 1\} 
\end{equation}
is uniformly bounded.

\item If for each $\xi\in \mathbb{R}$,
\begin{equation*}
    \|(i\xi+A)^{-1}\|_{\mathcal{L}(X)}\lesssim (1+|\xi|)^{\beta}(1/ \kappa(1+|\xi|)),
\end{equation*}
then the family
\begin{equation*}
\{\kappa(|\lambda|)^{- a}\|(\lambda+A)^{-1}(1+A)^{-\beta}\Phi(B)^{1+a}\|_{\mathcal{L}(X)}\mid \lambda \in i\mathbb{R}, |\lambda|\geq 1\} 
\end{equation*}
is uniformly bounded.
\end{enumerate}
\end{rm}
\end{proposition}
\begin{proof} 
\noindent~\textbf{(1)}. We proceed as in the proof of  \cite[Proposition 3.3]{nos}~(see also the \cite[Proposition~3.4-(2)]{rozendaal}). Set $\zeta_{\varepsilon}(\lambda):=\left(\Psi_{\varepsilon}(\lambda)\right)^{1+a}$; by \cite[Theorem 2.4.2]{haase}, $\left(\Psi_{\varepsilon}(B^{-1})\right)^{1+a}=\zeta_{\varepsilon}(B^{-1})$.
Fix $\theta \in (\omega_A,\pi)$ and let the path $\Gamma:=\{re^{i\theta} \mid r\in [0,\infty)\}\cup \{re^{-i\theta} \mid r\in [0,\infty)\}$
be oriented from $\infty e^{i\theta}$ to $\infty e^{-i\theta}$. 

Since $A+\frac{1}{2}\in\text{Sect}_X(\omega_A)$ and for each $\omega \in (\omega_A,\pi]$, $S_{\omega}\ni z\mapsto \eta_{\varepsilon}(z):= \frac{\Psi_\varepsilon\left(\left(\frac{1}{2}+z\right)^{-1}\right)^{1+a}}{\left(\frac{1}{2}+z\right)^{\beta}}$, then $\eta_\varepsilon (1/2+A)$ is still well-defined, since $\eta_\varepsilon$ is continuous, $0 \in \rho(1/2+ A)$,
and for some $\delta > 0, \eta_\varepsilon(z)= O(|z|^{-\delta})$, $|z| \to \infty$ for $z\in S_\omega$. 

Then, for each $\lambda\in i\mathbb{R}$, $\vert\lambda\vert\ge 1,$ and for each~$x\in X$,  
\begin{eqnarray}\label{36}
\nonumber (\lambda+A)^{-1}(1+A)^{-\beta}\Psi_{\varepsilon}(B^{-1})^{1+a}x&=& (\lambda+A)^{-1}\eta_\varepsilon\left(\frac{1}{2}+A\right)x\\
\nonumber&=& \frac{1}{2\pi i} \int_{\Gamma} \frac{(\lambda+A)^{-1}R\left(z,A+\frac{1}{2}\right)}{(\frac{1}{2}+z)^{\beta+(1+a)\varepsilon }S_{h_{\varepsilon}}\left((\frac{1}{2}+z)^{\varepsilon}\right)^{(1+a)}} xdz\\
\nonumber &=& \frac{1}{2\pi i} \int_{\Gamma} \frac{(\lambda+A)^{-1}}{(\frac{1}{2}+z)^{\beta+(1+a)\varepsilon}S_{h_{\varepsilon}}\left((\frac{1}{2}+z)^{\varepsilon}\right)^{(1+a)}(z+\lambda-\frac{1}{2})}x dz\\
\nonumber&+& \frac{1}{2\pi i} \int_{\Gamma} \frac{R\left(z,A+\frac{1}{2}\right)}{(\frac{1}{2}+z)^{\beta+(1+a)\varepsilon}S_{h_{\varepsilon}}\left((\frac{1}{2}+z)^{\varepsilon}\right)^{(1+a)}(z+\lambda-\frac{1}{2})}x dz\\
&=& \frac{1}{(1-\lambda)^{\beta+(1+a)\varepsilon}S_{h_{\varepsilon}}((1-\lambda)^{\varepsilon})^{(1+a)}}(\lambda+A)^{-1}x+T_\lambda x,
\end{eqnarray}
where
\begin{eqnarray*}
    T_\lambda x&:=& \frac{1}{2\pi i} \int_{\Gamma} \frac{R\left(z,A+\frac{1}{2}\right)}{(\frac{1}{2}+z)^{\beta+(1+a)\varepsilon}S_{h_{\varepsilon}}\left((\frac{1}{2}+z)^{\varepsilon}\right)^{(1+a)}(z+\lambda-\frac{1}{2})}x dz. 
\end{eqnarray*}

First, we will limit the first term of the equation~\eqref{36}. Note also  that since $S_{h_{\varepsilon}}$ is a Stieltjes function, one has by Proposition~\ref{prop21} that for each $z\in \Gamma$,
\begin{equation*}  S_{h_{\varepsilon}}\left(\left|\frac{1}{2}+z\right|^{\varepsilon}\right)\lesssim \left|S_{h_{\varepsilon}}\left(\left(\frac{1}{2}+z\right)^{\varepsilon}\right)\right|\lesssim S_{h_{\varepsilon}}\left(\left|\frac{1}{2}+z\right|^{\varepsilon}\right)
\end{equation*}
(note that $(\text{arg}(z+1/2))/2 \in \left(-\frac{\pi}{2},\frac{\pi}{2}\right)$). 

It follows from Theorem~\ref{lema21}-(ii) that for each $|z|\ge |z_0|$,
\begin{equation}\label{eq33}
\left|S_{h_{\varepsilon}}\left(\left|\frac{1}{2}+z\right|^{\varepsilon}\right)\right|\gtrsim |1+z|^{-\varepsilon} \ell\left(\left|\frac{1}{2}+z\right|\right)  
\end{equation}
and for sufficiently large values of $|z|$ and $|z_0|$ (with $z,z_0\in \Gamma$; here, $z_0$ is considered fixed), it follows from Proposition~\ref{cor2} (a) that for each $\gamma>0$,
\begin{equation}\label{34}
\ell\left(\left|\frac{1}{2}+z\right|\right) \gtrsim \left(\frac{|z_0|}{\left|\frac{1}{2}+z\right|}\right)^{\gamma}\ell (|z_0|).
\end{equation}

Finally, let $\gamma \in \left(0, \dfrac{\beta}{2(a+1)}\right)$ and note that the function  $z\mapsto (z+1/2)^{-\gamma(1+a)}R(z,A+1/2)$
 is integrable on $\Gamma$. By combining relation~\eqref{34}  with \cite[Lemma~A.1]{rozendaal}, one obtains for each $z\in \Gamma$ with sufficiently
large $|z|$ and each $\lambda \in i\mathbb{R}$,
\begin{eqnarray*}
 \frac{1}{|z+\frac{1}{2}|^{\beta+(1+a)\varepsilon}S_{h_\varepsilon}\left(\left|\frac{1}{2}+z\right|^{\varepsilon}\right)^{1+a}|z+\lambda-\frac{1}{2}|}&\lesssim& \frac{1}{|z+\frac{1}{2}|^{\beta-2\gamma(1+a)}|z+\lambda-\frac{1}{2}|}\lesssim \frac{1}{1+|\lambda|}.
\end{eqnarray*}

Therefore, by combining all the previous estimates with relation~\eqref{36}, one obtains 
\begin{eqnarray*}
\ell(|\lambda|)^{a} \|(\lambda+A)^{-1}(1+A)^{-\beta}\Psi_{\varepsilon}(B^{-1})^{1+a}\|_{\mathcal{L}(X)} 
  &\lesssim & \frac{1}{|1-\lambda|^{\beta}\ell(|1-\lambda|)}\|(\lambda+A)^{-1}\|_{\mathcal{L}(X)}+\frac{\ell(|\lambda|)^{a}}{1+|\lambda|},
\end{eqnarray*}
where we have used an estimate similar to the one presented in relation~\eqref{eq33} in order to control the first term on the right-hand side of relation~\eqref{36}.

Since for each $a\geq 0$, $\displaystyle{\lim_{|\lambda|\to \infty} \frac{\ell(|\lambda|)^a}{1+|\lambda|}=0}$ (see~Proposition~\ref{cor2}-(a)), the result follows.

\

\textbf{(2)}. It follows from the same arguments presented in the proof of item~\textbf{(1)}.
\end{proof}
\begin{remark}
\begin{rm}
Note that by relation~$\eqref{36}$, one has for each $\lambda \in i\mathbb{R}$  and for each $\varepsilon>0$ that
\begin{eqnarray*}
 \left\|\frac{(\lambda+A)^{-1}}{(1-\lambda)^{\beta+(1+a)\varepsilon}S_{h_{\varepsilon}}((1-\lambda)^{\varepsilon})^{1+a}}\right\|_{\mathcal{L}(X)} 
 &\lesssim& \|(\lambda+A)^{-1}(1+A)^{-\beta} \Psi_\varepsilon(B^{-\varepsilon})^{1+a}\|_{\mathcal{L}(X)}+\frac{1}{1+|\lambda|};
\end{eqnarray*}
thus, by assuming that condition~\eqref{eeq24} is valid, it follows from \cite[Proposition~2.4]{chill} and Lemma~\ref{lema21} that for sufficiently large values of $|\lambda|$,
\begin{eqnarray*}
  \left\|(\lambda+A)^{-1}\right\|_{\mathcal{L}(X)}&\lesssim& |(1-\lambda)^{\beta+\varepsilon(1+a)}S_{h_{\varepsilon}}((1-\lambda)^{\varepsilon})^{1+a}|\\
  &\lesssim & |1-\lambda|^{\beta+\varepsilon(1+a)}S_{h_{\varepsilon}}(|1-\lambda|^{\varepsilon})^{1+a}\\
  &\lesssim& |1-\lambda|^{\beta} \ell(|1-\lambda|).
\end{eqnarray*}
This shows that the converse of Proposition~\ref{prop3.1}-(1) is also valid. By  reasoning as before, one can also show that the converse of Proposition~\ref{prop3.1}-(2) is valid. 
\end{rm}
\end{remark}

In the following results, let $\sigma:=\frac{1}{\beta+1}\in (0,1)$ and set $f_\sigma(\lambda):=\lambda^{1-\sigma}\Psi_{\varepsilon}(1/\lambda)^{-\sigma}\in \mathcal{CBF}$, $g_{\sigma}(\lambda):=\lambda^{1-\sigma}\Phi(\lambda)^{-\sigma}\in \mathcal{CBF}$. Let also, for each $\tau>0$ and $q\in [1,\infty]$, $D_{f_{\sigma}}(\tau,q):=D_{f_{\sigma}(B)}(\tau,q)$.

\begin{lemma}\label{lemaaux}
\rm{Let $-A$ be the generator of a $C_0$-semigroup $(T(t))_{t\ge 0}$ defined in a Banach space $X$, and suppose that $\overline{\mathbb{C}_{-}}\subset \rho(A)$.
\begin{enumerate}[(1)]
\item If for each $\lambda \in \overline{\mathbb{C}_{-}}$,
$\|(\lambda+A)^{-1}\|_{\mathcal{L}(X)}\lesssim (1+|\lambda|)^{\beta} \ell (1+|\lambda|)$, then for each $n\in \mathbb{N}_0$, $k\in \{0,1,\ldots,n+1\}$ and $\xi\in \mathbb{R}$,
\[\|R(i\xi,A)^{k}\|_{\mathcal{L}(D_{f_{\sigma}}(((n+1)(\beta+1)+1,1),X)}\lesssim \dfrac{1}{(1+|\xi|)^{\frac{\beta}{\beta+1}}\ell(1+|\xi|)^{\frac{1}{\beta+1}}}.\]
\item If for each $\lambda \in \overline{\mathbb{C}_{-}}$, $\|(\lambda+A)^{-1}\|_{\mathcal{L}(X)}\lesssim \dfrac{(1+|\lambda|)^{\beta}}{\kappa(1+|\lambda|)}$, then for each $n\in \mathbb{N}_0$, $k\in \{0,1,\ldots,n+1\}$ and $\xi\in \mathbb{R}$,
\[\|R(i\xi,A)^{k}\|_{\mathcal{L}(D_{f_{\sigma}}(n(\beta+1)+1,1),X)}\lesssim (1+|\xi|)^{\frac{\beta^2}{\beta+1}}\ell(1+|\xi|)^{\frac{\beta}{\beta+1}}.\]
\end{enumerate}
}    
\end{lemma}
\begin{proof}
\noindent \textbf{(1)} Note that by Proposition~\ref{prop3.1}, one has for each $k\in\{1,2,\ldots, n\}$ that
\begin{eqnarray*}
\|R(i\xi,A)^{k}\|_{\mathcal{L}(D(B^{n\beta} [\Psi_{\varepsilon}(B^{-1})]^{-(n+1)}),X))} &\lesssim& \|R(i\xi,A)^{k}B^{-n\beta}\Psi_{\varepsilon} (B^{-1})^{n+1}\|_{\mathcal{L}(X)}\lesssim 1,
\end{eqnarray*}
and by replacing $\beta$ for $\beta+\dfrac{\beta}{\beta+1}$ in the proof of Proposition~\ref{prop3.1} (in case $a=\dfrac{1}{\beta+1}$), one gets
\begin{eqnarray*}
\|R(i\xi,A)\|_{\mathcal{L}(D(B^{\beta+\frac{\beta}{\beta+1}}[\Psi_{\varepsilon}(B^{-1})]^{-\frac{1}{\beta+1}}),X)}&\lesssim& \|R(i\xi,A)B^{-\beta-\frac{\beta}{\beta+1}}\Psi_{\varepsilon}(B^{-1})^{\frac{1}{\beta+1}}\|_{\mathcal{L}(X)}\\ 
&\lesssim&\frac{1}{(1+|\xi|)^{\frac{\beta}{\beta+1}}\ell(1+|\xi|)^{\frac{1}{\beta+1}}}.
\end{eqnarray*}

Thus, one has for each $k\in\{1,\ldots, n+1\}$ 
that
\begin{eqnarray*}
\nonumber \|R(i\xi,A)^{k}\|_{\mathcal{L}(D(C),X)}&\lesssim&\|R(i\xi,A)^{k}B^{-\beta(n+1)-\frac{\beta}{\beta+1}}\Psi_{\varepsilon}(B^{-1})^{n+1+\frac{1
}{\beta+1}}\|_{\mathcal{L}(X)}\\
\nonumber &\le& \|R(i\xi,A)^{k-1}B^{-n\beta}\Psi_{\varepsilon}(B^{-1})^{n+1}\|_{\mathcal{L}(X)}\\ 
\nonumber &\cdot& \|R(i\xi,A)B^{-\beta -\frac{\beta}{\beta+1}}\Psi_{\varepsilon}(B^{-1})^{\frac{1}{\beta+1}}\|_{\mathcal{L}(X)}\\ 
&\lesssim& \frac{1}{(1+|\xi|)^{\frac{\beta}{\beta+1}}\ell(1+|\xi|)^{\frac{1}{\beta+1}}},
\end{eqnarray*}
where $C:= B^{\beta(n+1)+\frac{\beta}{\beta+1}}[\Psi_{\varepsilon}(B^{-1})]^{-(n+1)-\frac{1}{\beta+1}}$. Therefore, the claim follows from~\eqref{inter},
since
 \begin{equation*}
D_{f_{\sigma}}((\beta+1)(n+1)+1,1)\subset D((f_{\sigma}(B))^{(\beta+1)(n+1)+1})=D(B^{\beta(n+1)+\frac{\beta}{\beta+1} }\Psi_{\varepsilon}(B^{-1})^{n+\frac{1}{\beta+1}}).
\end{equation*}

\noindent \textbf{(2)} Case $n=0$. It follows from Proposition~\ref{fsec} that for each $x\in X\setminus\{0\}$ and each $\xi \in \mathbb{R}$,
\begin{eqnarray}\label{eq35}
&&\nonumber\|R(i\xi,A)B^{-\frac{\beta}{\beta+1}}\Psi_{\varepsilon}(B^{-1})^{\frac{1}{\beta+1}}x\|_X\\
\nonumber&\lesssim&\|R(i\xi,A)B^{-\frac{\beta}{\beta+1}}x\|_X \left(\Psi_\varepsilon\left(\frac{\|R(i\xi,A)B^{-1-\frac{\beta}{\beta+1}}x\|_X}{\|R(i\xi,A)B^{-\frac{\beta}{\beta+1}}x\|_X}\right)\right)^{\frac{1}{\beta+1}}
\\
\nonumber &\lesssim& \|R(i\xi,A)B^{-\frac{\beta}{\beta+1}}x\|^{1-\frac{\varepsilon}{\beta+1}}_X \|\|R(i\xi,A)B^{-1-\frac{\beta}{\beta+1}}x\|^{\frac{\varepsilon}{\beta+1}}_X \left({S_{h_{\varepsilon}}\left(\dfrac{\|R(i\xi,A)B^{-\frac{\beta}{\beta+1}}x\|^{\varepsilon}_X}{ \|R(i\xi,A)B^{-1-\frac{\beta}{\beta+1}}x\|^{\varepsilon}_X}\right)}\right)^{-\frac{1}{\beta+1}}\\
&\lesssim& \dfrac{\|R(i\xi,A)B^{-\frac{\beta}{\beta+1}}x\|_X}{\left(\ell\left(\dfrac{\|R(i\xi,A)B^{-\frac{\beta}{\beta+1}}x\|_X}{ \|R(i\xi,A)B^{-1-\frac{\beta}{\beta+1}}x\|_X}\right)\right)^{\frac{1}{\beta+1}}} \lesssim \dfrac{\|R(i\xi,A)B^{-\frac{\beta}{\beta+1}}x\|_X}{\left(\ell\left(\dfrac{\|R(i\xi,A)B^{-\frac{\beta}{\beta+1}}x\|_X}{ (1+|\xi|)^{\beta-1-\frac{\beta}{\beta+1}}\ell(1+|\xi|) \|x\|_X}\right)\right)^{\frac{1}{\beta+1}}},
\end{eqnarray}
where we have used relation~\eqref{eq33} in the third inequality. 

One has for each $x\in X \setminus\{0\}$ that $\dfrac{\|R(i\xi,A)B^{-\frac{\beta}{\beta+1}}x\|_X}{(1+|\xi|)^{\beta-1-\frac{\beta}{\beta+1}}\ell(1+|\xi|) \|x\|_X}\lesssim (1+|\xi|)$ (see \cite[Proposition~3.2]{rozendaal}), and so it follows from Proposition~\ref{cor2}(a) that
\begin{equation}\label{eq36}
    \frac{(\ell(1+|\xi|))^{\frac{1}{\beta+1}}}{\left(\ell\left(\dfrac{\|R(i\xi,A)B^{-\frac{\beta}{\beta+1}}x\|_X}{ (1+|\xi|)^{\beta-1-\frac{\beta}{\beta+1}}(\ell(1+|\xi|) )|x\|_X}\right)\right)^{\frac{1}{\beta+1}}}\lesssim \frac{(1+|\xi|)^{\beta-\frac{\beta}{\beta+1}}\ell(1+|\xi|)\|x\|_X}{\|R(i\xi,A)B^{-\frac{\beta}{\beta+1}}x\|_X}.
\end{equation}
By combining relations~\eqref{eq35} and~\eqref{eq36}, one gets
\begin{equation*}
    \|R(i\xi,A)B^{-\frac{\beta}{\beta+1}}\Psi_{\varepsilon}(B^{-1})^{\frac{1}{\beta+1}}\|_{\mathcal{L}(X)} \lesssim (1+|\xi|)^{\beta-\frac{\beta}{\beta+1}}\ell(1+|\xi|)^{1-\frac{1}{\beta+1}}=(1+|\xi|)^{\frac{\beta^2}{\beta+1}}\ell(1+|\xi|)^{\frac{\beta}{\beta+1}}.
\end{equation*}
Case $n\in \mathbb{N}$. It follows from \cite[Theorem~2.4.2]{haase} and from Proposition~\ref{prop3.1} (with $a=\dfrac{1}{(\beta+1)n}$) that 
\begin{eqnarray*}
&&\|R(i\xi,A)^{n+1}B^{-n\beta-\frac{\beta}{\beta+1}}\Psi_{\varepsilon}(B^{-1})^{n+\frac{1}{\beta+1}}\|_{\mathcal{L}(X)}\lesssim\|R(i\xi,A)\|\|R(i\xi,A)B^{-\beta}\Psi_{\varepsilon}(B^{-1})^{1+\frac{1}{(\beta+1)n}}\|^n\\
&\lesssim& \frac{\|R(i\xi,A)B^{-\frac{\beta}{\beta+1}}\|_{\mathcal{L}(X)}}{\ell(1+|\xi|)^{\frac{1}{\beta+1}}}\lesssim (1+|\xi|)^{\beta-\frac{\beta}{\beta+1}} \ell(1+|\xi|)^{\frac{\beta}{\beta+1}};
\end{eqnarray*}
therefore, one has for each $k\in\{1,\ldots,n+1\}$ that
\begin{eqnarray*}
   \|R(i\xi,A)^{k}B^{-\beta n-\frac{\beta}{\beta+1}}\Psi_{\varepsilon}(B^{-1})^{n+\frac{1}{\beta}}\|_{\mathcal{L}(X)}
    &\lesssim& (1+|\xi|)^{\beta-\frac{\beta}{\beta+1}} \ell(1+|\xi|)^{\frac{\beta}{\beta+1}}.
\end{eqnarray*}
Therefore, the claim follows from~\eqref{inter},
since

\begin{equation*}
    D_{f_{\sigma}}((\beta+1)n+1,1)\subset D((f_{\sigma}(B))^{(\beta+1)n+1})=D(B^{\beta n+\frac{\beta}{\beta+1}}\Psi_{\varepsilon}(B^{-1})^{n+\frac{1}{\beta+1}}).
\end{equation*}
\end{proof}

\begin{proposition}
 \label{Prop3.1}
\rm{Let $-A$ be the generator of a $C_0$-semigroup $(T(t))_{t\ge 0}$ defined in a Banach space $X$ with Fourier type $p\in[1,2]$. Let $n\in \mathbb{N}_0$, and suppose that $\overline{\mathbb{C}_{-}}\subset \rho(A)$. 
\begin{enumerate}[(1)]
\item If for each $\lambda \in \overline{\mathbb{C}_{-}}$,
$\|(\lambda+A)^{-1}\|_{\mathcal{L}(X)}\lesssim (1+|\lambda|)^{\beta} \ell (1+|\lambda|)$, then for each $k\in \{0,1,\ldots,n+1\}$ and each $q\in [1,\infty]$,
\begin{equation*}
    \sup\{\|R(i\xi,A)^{k}\|_{\mathcal{L}\left(D_{f_{\sigma}}((n+1)(\beta+1),q),X\right)}\mid |\xi|\geq 1\}<\infty.
\end{equation*}

\item If for each $\lambda \in \overline{\mathbb{C}_{-}}$, $\|(\lambda+A)^{-1}\|_{\mathcal{L}(X)}\lesssim \dfrac{(1+|\lambda|)^{\beta}}{\kappa(1+|\lambda|)}$, then for each $k\in \{0,1,\ldots,n+1\}$ and each $q\in [1,\infty]$,
\begin{equation*}
    \sup\{\|R(i\xi,A)^{k}\|_{\mathcal{L}\left(D_{g_{\sigma}}((n+1)(\beta+1),q),X\right)}\mid |\xi|\geq 1\}<\infty.
\end{equation*}

\end{enumerate}
 }
\end{proposition}

\begin{proof} 
Since the case $k=0$ is trivial, we only consider the case $k>0$.

\noindent~\textbf{(1)}.
Let $\xi\in \mathbb{R}$ be such that $|\xi|\ge 1$. The result is a consequence of the identity
\begin{equation*}
D_{f_{\sigma}}((n+1)(\beta+1),q)=(D_{f_{\sigma}}(n(\beta+1)+1,1),D_{f_{\sigma}}((n+1)(\beta+1) +1,1))_{\frac{\beta}{\beta+1}, q}
\end{equation*}
(which is valid for each $n\in \mathbb{N}_0$ and each $q\in [1,\infty]$), Lemma~\ref{lemaaux} and Theorem~\ref{inter1}.

\noindent~\textbf{(2)}. It follows from the same arguments presented in the proof of item~\textbf{(1)}.

\end{proof}

The following result is a version of \cite[Theorem~3.2]{deng} for our setting.

\begin{theorem}\label{teo32}
\rm{Let $-A$ be the generator of a $C_0$-semigroup $(T(t))_{t\ge 0}$ defined in a Banach space $X$ with Fourier type $p\in[1,2]$. Suppose that $\overline{\mathbb{C}_{-}}\subset \rho(A)$.
 \begin{enumerate}[(i)]
\item Assume that for each $\lambda \in \overline{\mathbb{C}_{-}}$,
 \begin{equation*}
\|R(\lambda,A)\|_{\mathcal{L}(X)}\lesssim (1+|\lambda|)^{\beta}\ell(1+|\lambda|).
 \end{equation*}
 
 Let $\varepsilon>0$, $\gamma>0$, $p\in [1,\infty)$, $s\in (0,1/p)$, and suppose that there exist $n\in \mathbb{N}_0$ and
$q\in [1,\infty]$ such that
\begin{equation*}
T_{R(i\cdot,A)^k}:B^{s}_{p,p}(\mathbb{R}, D_{f_{\sigma}}(\gamma,p))\rightarrow L^q(\mathbb{R},X)
\end{equation*}
 is bounded for each $k\in \{n-1,n,n+1\}\cap \mathbb{N}$.  Then, there exists $C_{n,\varepsilon}\geq 0$ such that for each $t\geq 1$ and each $x\in D_{f_\sigma}(\gamma,p)$,
\begin{equation*}
  \|T(t)B^{-\sigma s}\Psi_\varepsilon(B^{-1})^{-\sigma s}x\|_X\leq C_{n,\varepsilon} t^{-n} \|x\|_{D_{f_\sigma}(\gamma+s,p)}.  
\end{equation*} 

\item Suppose that for each $\lambda \in \overline{\mathbb{C}_{-}}$,
 \begin{equation*}
\|R(\lambda,A)\|_{\mathcal{L}(X)}\lesssim (1+|\lambda|)^{\beta}(1/\kappa(1+|\lambda|)).
 \end{equation*}
 Let $\gamma>0$, $p\in [1,\infty)$, $s\in (0,1/p)$, and suppose that there exist $n\in \mathbb{N}_0$ and $q\in [1,\infty]$ such that
\begin{equation*}
T_{R(i\cdot,A)^k}:B^{s}_{p,p}(\mathbb{R}, D_{g_\sigma}(\gamma,p))\rightarrow L^q(\mathbb{R},X)
\end{equation*}
 is bounded for each $k\in \{n-1,n,n+1\}\cap \mathbb{N}$.  Then, there exists $\tilde{C}_{n}\geq 0$ such that for each $t\geq 1$ and each $x\in D_{g_{\sigma}}(\gamma,p)$,
\begin{equation*}
  \|T(t)B^{-\sigma s}\Phi(B^{-1})^{-\sigma s}x\|_X\leq \tilde{C}_{n} t^{-n} \|x\|_{D_{g_\sigma}(\gamma+s,p)}.  
\end{equation*} 

 \end{enumerate}
}
\end{theorem}

\begin{proof}
\noindent\textbf{(1)} 
We want to show that $\|T (t)x\|_X \le c_nt^{-n}\|x\|_{D_{f_\sigma}(\gamma+s,p)}$ for each $t\ge 1$ and each $x\in D_{f_{\sigma}}(\gamma+s,p)$. 

Firstly, note that for each $l\in \mathbb{N}$ such that $l>\gamma+s$, $D(A^{l})=D(B^l)$ (see~\cite[Proposition~3.1.9 (a)]{haase}) is dense in $D_{f_{\sigma}}(\gamma+s,p)$. Namely, fix $l$ satisfying such conditions; we argue as follows: for each $m \in \mathbb{N}$, $x \in D([f_\sigma(B)]^m)$ and $\delta \in [0, 1-\sigma)$, there exists $y\in X$
 such that $x = \psi_{\varepsilon}(B^{-1})^{m\sigma} B^{-\delta m} B^{m(\sigma+\delta)-m} y$. Since 
$\psi_{\varepsilon}(B^{-1})^{m\sigma} B^{-\delta m} \in \mathcal{L}(X)$, it follows that  
$x \in D(B^{m-(\delta+\sigma)m})$. Now, choose $m\in \mathbb{N}$ such that $m > \dfrac{l}{1-(\delta+\sigma)}>\gamma+s$. Then, one has the inclusions  
$D([f_\sigma(B)]^m) \subset D(B^{l}) \subset D_{f_{\sigma}}(\gamma+s,p)$, and since $D([f_\sigma(B)]^m)$ is dense in $D_{f_{\sigma}}(\gamma+s,p)$ (see \cite[Theorem 6.6.1]{haase}), the result follows.

Next, we follow the steps presented in the proof of \cite[Theorem~3.2]{deng}, with the necessary modifications: replace $D(A)$ with $D(f_\sigma(B))$, $g(t)= t^n\textbf{1}_{(0,\infty)}(t)T(|t|)x$ with $g(t):= t^n\textbf{1}_{(0,\infty)}(t)T(|t|)B^{-\sigma s}$ $\Psi_\varepsilon(B^{-1})^{-\sigma s}x$ and $x\in D(A^l)$ for $l\in \mathbb{N}$ sufficiently large.  Furthermore, note that the generator of $(T(t))_{t\ge 0}$ restricted to $D_{f_\sigma}(\gamma, p)$ is the part of $A$ in $D_{f_{\sigma}}(\gamma,p)$. More specifically, it suffices to show that 
\begin{equation*}
    \|f\|_{B^s_{p,p}(\mathbb{R}, D_{f_\sigma}(\gamma,p))}\lesssim \|x\|_{(D_{f_\sigma}(\gamma,p), D(B|D_{f_\sigma}))_{s,p}}\lesssim\|x\|_{(D_{f_\sigma}(\gamma,p), D_{f_\sigma}(\gamma+1,p))_{s,p}}\lesssim \|x\|_{D_{f_\sigma(\gamma+s,p)}};
\end{equation*}
the first inequality comes from \cite[Proposition~2.4]{deng}, and the third one comes from \cite[Theorem~1.10.2]{Hans}, so it remains to show the second one.

In order to do that, we begin showing that $D_{f_\sigma}(\gamma+1,p)$ is isomorphic to $D(B|_Y)$, where $Y:= D_{f_\sigma}(\gamma,p)$. Namely, let $W := D(B|_Y) = \{ x \in D(B) \cap Y \mid Bx \in Y \}$ and $Z := D_{f_\sigma}(\gamma+1,p) = \{ x \in D(f_\sigma(B)) \mid f_\sigma(B)x \in Y \}$ (this last identity comes from \cite[Proposition~2.6.5]{haase}; indeed,  $f_\sigma (B)$ commutes with $R(\lambda,A)$ for each $\lambda \in \rho (A)$, so relation~(6.2) in \cite{haase} is valid and $A|_{D_{f_{\sigma}}}(\gamma,p)$ is sectorial; then, $f_{\sigma}(A|_{D_{f_{\sigma}(\gamma,p)}})=f_\sigma(A)|_{D_{f_{\sigma}(\gamma,p)}}$ and $D(f_\sigma(A)|_{D_{f_{\sigma}(\gamma,p)}})=D(f_\sigma(B)|_{D_{f_{\sigma}(\gamma,p)}})=D_{f_\sigma}(\gamma+1,p)$). 

Define $T: Z \rightarrow W$ by $T(x) = B^{-1} f_\sigma(B) x$. Note that:
\begin{enumerate}[(i)]
    \item \textit{$T$ is well-defined.}
    For each $x \in Z$, note that $B^{-1}(f_\sigma(B)x) \in D(B) \cap Y$ and $B(B^{-1}(f_\sigma(B)x)) = f_\sigma(B)x \in Y$. Therefore, $T(x) \in W$. 

    \item $T \in \mathcal{L}(Z,W)$.
    \begin{eqnarray*}
        \|T(x)\|_W &=& \|B^{-1}f_\sigma(B)x\|_Y + \|B(B^{-1}f_\sigma(B)x)\|_Y 
        \lesssim\|f_\sigma(B)x\|_Y \\
        &\le& \|x\|_Y+\|f_\sigma(B)x\|_Y=\|x\|_Z.
    \end{eqnarray*}

    \item $T^{-1}:W\rightarrow Z$, $T^{-1}(x) = f_\sigma(B)^{-1} B x$, is such that $T^{-1} \in \mathcal{L}(W,Z)$. Namely, since $f_\sigma(B)^{-1}\in \mathcal{L}(X)$, one has 
    \begin{equation*}
        \|T^{-1}(x)\|_Z = \|f_\sigma(B)^{-1} B x\|_Y + \|Bx\|_Y \lesssim \|Bx\|_Y \lesssim \|x\|_W.
    \end{equation*}

\end{enumerate}

Therefore, $D_{f_\sigma}(\gamma+1,p)$ is isomorphic to $D(B|Y)$. Now, relation $\|x\|_{(D_{f_\sigma}(\gamma,p), D(B|D_{f_\sigma}))_{s,p}}\lesssim\|x\|_{(D_{f_\sigma}(\gamma,p), D_{f_\sigma}(\gamma+1,p))_{s,p}}$ is a direct consequence of Definitions 1.1 and 1.2 in \cite{Lunardi}.

\noindent\textbf{(2)}~It follows from the same arguments presented in the proof of item~\textbf{(1)}.
\end{proof}

The following result is an extension of \cite[Corollary 1.14]{nos} to a larger class of functions.

\begin{theorem}\label{THIL}
  \rm{Let $-A$ be the generator of a $C_0$-semigroup $(T(t))_{t\ge 0}$ defined in a Hilbert space $X$. Suppose that $\overline{\mathbb{C_{-}}}\subset \rho(A)$. Let $\rho>0$ and set $\tau:= (\rho+1)(\beta+1)$.
\begin{enumerate}[(i)]
\item Suppose that for each $\lambda \in \overline{\mathbb{C}_{-}}$,
 \begin{equation*}
\|R(\lambda,A)\|_{\mathcal{L}(X)}\lesssim (1+|\lambda|)^{\beta}\ell(1+|\lambda|).
 \end{equation*}

 Then, for each $\varepsilon>0$, there exists $C_{\rho,\varepsilon} \ge 0$ such that for each $t\ge 1$ and each $x\in D_{f_{\sigma}}(\tau,p)$,
 \begin{equation*}
     \|T(t)x\|_{X}\le C_{\rho,\varepsilon}t^{-\rho}\|x\|_{D_{f_{\sigma}(\tau,p)}}.
 \end{equation*}

\item Suppose that for each $\lambda \in \overline{\mathbb{C}_{-}}$,
 \begin{equation*}
\|R(\lambda,A)\|_{\mathcal{L}(X)}\lesssim (1+|\lambda|)^{\beta}(1/\kappa(1+|\lambda|)).
 \end{equation*}

Then, there exists $\tilde{C}_{\rho} \ge 0$ such that for each $t\ge 1$ and each $x\in D_{g_{\sigma}}(\tau,p)$, 
 \begin{equation*}
\|T(t)x\|_{X}\le \tilde{C}_{\rho}t^{-\rho}\|x\|_{D_{g_{\sigma}}(\tau,p)}.
 \end{equation*}
\end{enumerate}
}
\end{theorem}

\begin{proof}
\noindent~\textbf{(i)}. Let $n\in \mathbb{N}_0$ and let $\psi\in C^{\infty}_c(\mathbb{R})$ be such that $\psi \equiv 1$ in $[-1,1]$; then,
\begin{equation*}
   [\xi\mapsto \psi(\xi)R(i\xi,A)^k]\in L^{1}(\mathbb{R}, \mathcal{L}(D_{f_\sigma}((n+1)(\beta+1),p),X))\subset \mathcal{M}_{1,\infty}(\mathbb{R}, \mathcal{L}(D_{f_\sigma}((n+1)(\beta+1),p),X)),
\end{equation*}
and by Proposition~\ref{Prop3.1}, 
\begin{eqnarray*}
   [\xi\mapsto (1-\psi(\xi))R(i\xi,A)^k]&\in& L^{\infty}(\mathbb{R}, \mathcal{L}(D_{f_\sigma}((n+1)(\beta+1),p),X))\\
   &\subset &\mathcal{M}_{2,2}(\mathbb{R}, \mathcal{L}(D_{f_\sigma}((n+1)(\beta+1),p),X)).
\end{eqnarray*}

Therefore, it follows from \cite[Theorem~4.6]{rozendaal} (note that $D_{f_\sigma}((n+1)(\beta+1),p),X)$ satisfies the hypotheses stated there) that there exists $C_{n,\varepsilon}\geq 0$ such that for each $t\geq 1$,
\begin{equation*}
    \|T(t)\|_{\mathcal{L}(D_{f_\sigma}((n+1)(\beta+1),p),X)}\leq C_{n,\varepsilon} t^{-n}.
\end{equation*}

Now, for each $\rho>0$, let $\theta \in (0,1)$ and $n\in \mathbb{N}_0$ be such that $\rho=(1-\theta)n+\theta (n+1)$; it follows from Theorem~\ref{inter} that there exists $C_{\rho,\varepsilon}\geq 0$ such that for each $t\ge 1$,
\begin{equation*}
    \|T(t)\|_{\mathcal{L}(D_{f_\sigma}((\rho+1)(\beta+1),p),X)}\leq C_{\rho,\varepsilon} t^{-\rho}.
\end{equation*}

\

\textbf{(2)}. It follows from the same arguments presented in the proof of item~\textbf{(1)}.

\end{proof}

The following result is a version of \cite[Theorem 3.3]{deng} for the case where the norm of the resolvent is bounded by a regularly varying function of index $\beta$.

\begin{theorem}\label{teo31}
 \rm{Let $-A$ be the generator of a $C_0$-semigroup $(T(t))_{t\ge 0}$ defined in a Banach  space $X$ with Fourier type $p\in [1,2]$. Suppose that $\overline{\mathbb{C_{-}}}\subset \rho(A)$. Let $\rho\ge 0$, let $r\in(0,\infty]$ be such that $\frac{1}{r}=\frac{1}{p}-\frac{1}{p'}$, and set $\tau:= (\rho+1)(\beta+1)+\frac{1}{r}$. 
 \begin{enumerate}[(i)]
\item Suppose that for each $\lambda \in \overline{\mathbb{C}_{-}}$,
 \begin{equation*}
\|R(\lambda,A)\|_{\mathcal{L}(X)}\lesssim (1+|\lambda|)^{\beta}\ell(1+|\lambda|).
 \end{equation*}
 Then, for each $\varepsilon>0$, there exists $C_{\rho,\varepsilon} \ge 0$ such that for each $t\ge 1$ and each $x\in D_{f_\sigma}(\tau,p)$,
 \begin{equation*}
     \|T(t)B^{-\sigma}\Psi_{\varepsilon}(B^{-1})^{-\sigma}x\|_{X}\le C_{\rho,\varepsilon}t^{-\rho}\|x\|_{D_{f_\sigma}(\tau,p)}.
 \end{equation*}

\item Suppose that for each $\lambda \in \overline{\mathbb{C}_{-}}$,
 \begin{equation*}
\|R(\lambda,A)\|_{\mathcal{L}(X)}\lesssim (1+|\lambda|)^{\beta}(1/\kappa(1+|\lambda|)).
\end{equation*}
 Then, there exists $\tilde{C}_{\rho} \ge 0$ such that for each $t\ge 1$ and each $x\in D_{g_{\sigma}}(\tau,p)$, 
 \begin{equation*}
\|T(t)B^{-\sigma} \Phi(B)^{-\sigma}x\|_{X}\le \tilde{C}_{\rho}t^{-\rho}\|x\|_{D_{g_{\sigma}}(\tau,p)}.
 \end{equation*}
\end{enumerate}
}
\end{theorem}

\begin{proof}
\textbf{(i)}  We proceed as in the proof of \cite[Theorem~3.3]{deng}.

 \noindent~1.  \textbf{Case} $\rho\in \mathbb{N}_0$. Set $n=\rho$. It follows from Proposition~\ref{prop3.1} that for each $k\in\{0, \ldots, n+1\}$, 
\begin{equation*}
\sup\{\|R(i\xi,A)^{k}\|_{\mathcal{L}(D_{f_\sigma}((n+1)(\beta+1),p),X)}\mid \xi\in \mathbb{R}\}<\infty,
\end{equation*}
and so, by Proposition~\ref{prop22},
\begin{equation*}
    T_{R(i\cdot,A)^{k}}: B^{\frac{1}{r}}_{p,p}(\mathbb{R}; D_{f_\sigma}((n+1)(\beta+1),p),X))\rightarrow L^{p'}(\mathbb{R}; X)
\end{equation*}
is bounded for each $k\in\{0,\ldots,n+1\}$.

In case $p\in[1,2)$,  it follows from Theorem~\ref{teo32} (with $s=1/r$ and $q=p'$) that there exists $C_{n,\varepsilon} \ge 0$ such that for each $t\ge 1$ and each $x\in D_{f_\sigma}((n+1)(\beta+1)+1/r,p)$,
 \begin{equation*}
\|T(t)B^{-\sigma}\Psi_{\varepsilon}(B^{-1})^{-\sigma}x\|_{X}\le C_{n,\varepsilon}t^{-n}\|x\|_{D_{f_{\sigma}}((n+1)(\beta+1)+1/r,p)}.
\end{equation*}

\
 
\noindent~2. \textbf{Case} $\rho\in\mathbb{R}_+\setminus\mathbb{N}$. Set $n=\lfloor \rho\rfloor$ and let $\theta \in (0,1)$ be such that $\rho=(1-\theta)n+\theta (n+1)$. It follows from Theorem~\ref{inter1} that for each $t\geq 1$,
\begin{equation}\label{rho}
    \|T(t)B^{-\sigma}\Psi_{\varepsilon}(B^{-1})^{-\sigma}\|_{\mathcal{L}(D_{f_{\sigma}}((\rho+1)(\beta+1)+1/r,p),X)} \lesssim t^{-\rho}.
\end{equation}

\noindent~\textbf{(ii)} The proof follows from the same ideas presented in the proof of item~(i). 
\end{proof}

The following result, a direct consequence of Theorem~\ref{teo31}, will be used in the proof of Theorem~\ref{theo4.5}.

\begin{corollary}\label{RFINAL}

\rm{Let $-A$ be the generator of a $C_0$-semigroup $(T(t))_{t\ge 0}$ defined in a Banach  space $X$ with Fourier type $p\in [1,2]$. Suppose that $\overline{\mathbb{C_{-}}}\subset \rho(A)$. Let $\rho>0$, let $r\in(0,\infty]$ be such that $\frac{1}{r}=\frac{1}{p}-\frac{1}{p'}$.
 \begin{enumerate}[(i)]
\item Suppose that for each $\lambda \in \overline{\mathbb{C}_{-}}$,
 \begin{equation*}
\|R(\lambda,A)\|_{\mathcal{L}(X)}\lesssim (1+|\lambda|)^{\beta}\ell(1+|\lambda|).
 \end{equation*}

 Then, for each $\varepsilon>0$, there exists $C_{\rho,\varepsilon} \ge 0$ such that for each $t\ge 1$,
 \begin{equation*}
     \|T(t)B^{-\beta(\rho+1)-1/r}\Psi_\varepsilon(B^{-1})^{\rho+1}\|_{\mathcal{L}(X)}\le C_{\rho,\varepsilon}t^{-\rho}.
 \end{equation*}

\item Suppose that for each $\lambda \in \overline{\mathbb{C}_{-}}$,
 \begin{equation*}
\|R(\lambda,A)\|_{\mathcal{L}(X)}\lesssim (1+|\lambda|)^{\beta}(1/\kappa(1+|\lambda|)).
\end{equation*}

 Then, there exists $\tilde{C}_{\rho} \ge 0$ such that for each $t\ge 1$, 
 \begin{equation*}
\|T(t)B^{-\beta(\rho+1)-1/r} \Phi(B)^{\rho+1}\|_{X}\le \tilde{C}_{\rho}t^{-\rho}.
 \end{equation*}
\end{enumerate}
}
\end{corollary}

\begin{proof}
\noindent~\textbf{(i)}~Note that for each $\rho>0$, relation~\eqref{rho} also holds with $D_{f_{\sigma}} (\tau, p)$ replaced with $D(f_\sigma(B)^{\tau})$, where $\tau=(\beta+1)(\rho+1)+1/r$ (in the proof of Case 2 of Theorem~\ref{teo31}, the final estimate follows from Theorem~\ref{inter1}, so one can replace
$D_{f_\sigma}(\tau,p)$ in the result of Theorem~\ref{teo31} with $D_{f_\sigma}(\tau,q)$, for each $q \in [1, \infty]$,
and then the result for $D(f_{\sigma}(B)^\tau)$ follows from relation~\eqref{inter}). Namely, for each $t \geq 1$, one has
    \begin{eqnarray}\label{rel}
 \nonumber  t^{-\rho}&\gtrsim& \|T(t)B^{-\beta(\rho+1)-\frac{\beta}{r(\beta+1)}-\frac{1}{(\beta+1)r}}\Psi_\varepsilon(B^{-1})^{-\frac{1}{(\beta+1)r}}\Psi_\varepsilon(B^{-1})^{\rho+1+\frac{1}{(\beta+1)r}}\|_{\mathcal{L}(X)}\\
   &=& \|T(t)B^{-\beta(\rho+1)-\frac{1}{r}}\Psi_\varepsilon(B^{-1})^{\rho+1}\|_{\mathcal{L}(X)}
    \end{eqnarray}

\noindent~\textbf{(ii)} The proof follows from the same ideas presented in the proof of item~(i). 
\end{proof}

Now, we prove Theorem~\ref{theo4.5}.

\begin{theorem}
\begin{rm}
Let $\beta>0$, $(T(t))_{t\geq 0}$ be a $C_0$-semigroup with generator $-A$  defined in a Banach space $X$ with Fourier type $p\in [1,2]$. Suppose that $\overline{\mathbb{C}_{-}}\subset \rho(A)$ and let $r\geq 0$ be such that $\frac{1}{r}=\frac{1}{p}-\frac{1}{p'}$. 
\begin{enumerate}[(i)]
\item Assume that for each $\lambda \in \mathbb{C}$ with $\text{Re}(\lambda)\le 0$,
\[\|(\lambda+A)^{-1}\|_{\mathcal{L}(X)}\lesssim (1+|\lambda|)^{\beta}\ell(1+|\lambda|),\]
where $\ell$ is an increasing and slowly varying function. Then, for each $\tau > 0$ such that $\tau>\beta+\frac{1}{r}$, there exist $c_\tau>0$ and $t_0\geq 1$ so that for each $t\geq t_0$,
 \begin{equation}\label{eq16}
 \|T(t)(1+A)^{-\tau}\|_{\mathcal{L}(X)}\leq c_{\tau}t^{1-\frac{\tau-r^{-1}}{\beta}}(\ell(t^{1/\beta}))^{\frac{\tau-r^{-1}}{\beta}}.
\end{equation}
\item Assume that for each $\lambda \in \mathbb{C}$ with $\text{Re}(\lambda)\le 0$,
\[\|(\lambda+A)^{-1}\|_{\mathcal{L}(X)}\lesssim (1+|\lambda|)^{\beta}(1/\kappa(1+|\lambda|)),\]
where $\kappa$ is an increasing and slowly varying function. Then, for each $\tau > 0$ such that $\tau>\beta+\frac{1}{r}$, there exist $\tilde{c}_\tau>0$ and $t_0\geq 1$ so that for each $t\geq t_0$,
 \begin{equation*}
 \|T(t)(1+A)^{-\tau}\|_{\mathcal{L}(X)}\leq \tilde{c}_{\tau}t\left(\frac{1}{t\kappa(t^{1/\beta})}\right)^{\frac{\tau-r^{-1}}{\beta}}.
\end{equation*}
\end{enumerate}
\end{rm}
\end{theorem}

\begin{proof} 
\noindent \textbf{(i)} We proceed as in the proof of \cite[Theorem 1.13]{nos}. The result is equivalent to the following
statement: for each $s\geq 0$, there exist $C_{s}>0$ and $t_0\geq 1$ such that for each $t\geq t_0$,
\begin{equation*}
\|T(t)(1+A)^{-\nu}\|_{\mathcal{L}(X)}\le C_{s} t^{-s} \ell(t)^{s+1},
\end{equation*}
where $\nu:=\beta(s+1)+1/r$   for $p\neq 2$ and $\nu:=\beta(s+1)$ otherwise. Set $m:=\lfloor s+1 \rfloor$ and $\delta:=s+1-\lfloor s+1 \rfloor\in [0,1)$.

\noindent \textbf{Case $\delta>0$.}

\noindent \textbf{Step 1: removing $\delta>0$}. Let $\theta=\theta(s)\in (0,1)$ be such that $s\ge \theta>0$ and set $\varepsilon:=\frac{\min\{1,\beta\}\theta}{2\delta}>0$. 
For each $x\in X$ and each $t\ge 1$, it follows from the functional calculus defined by~\eqref{cbf} (note that $(S_{h_{\varepsilon\delta}}(1/\lambda^{\varepsilon\delta}))^{\delta}\in \mathcal{CBF}$ and that $\Psi_{\varepsilon\delta}(B^{-1})(S_{h_{\varepsilon\delta}}(B^{-\varepsilon\delta}))^{\delta}=B^{-\varepsilon\delta}$) that
\begin{eqnarray}\label{ig}
\nonumber && T(t)B^{-\nu+\varepsilon\delta}B^{-\varepsilon\delta }\Psi_{\varepsilon\delta}(B^{-1})^{m}x=T(t)B^{-\nu+\varepsilon\delta}(S_{h_{\varepsilon\delta}}(B^{\varepsilon\delta}))^{\delta}\Psi_{\varepsilon\delta}(B^{-1})^{s+1}x\\ &=&\int_{0}^{\infty}B^{-\varepsilon\delta}(\lambda+B^{-\varepsilon\delta})^{-1}T(t)B^{-\nu+\varepsilon\delta}\Psi_{\varepsilon\delta}(B^{-1})^{s+1}x d\mu(\lambda).
\end{eqnarray}

Let 
\begin{equation}\label{tau}
    \tau:= \dfrac{\|T(t)B^{-\nu}\Psi_{\varepsilon\delta}(B^{-1})^{s+1}\|_{\mathcal{L}(X)}}{t^{-s+\varepsilon\delta/\beta}}>0;
\end{equation}
since $B^{-\varepsilon\delta}$ is a sectorial operator and since, for each $t\geq 0$, $T(t)$ commutes with $B^{-\varepsilon\delta}$, it follows that for each $t\geq 0$ and each $x \in X$,
\begin{eqnarray}
\label{eqq337}
\nonumber&&\left\|T(t)B^{-\nu+\varepsilon\delta}\Psi_{\varepsilon\delta}(B^{-1})^{s+1}\int_{0+}^{\tau}B^{-\varepsilon\delta}(\lambda+B^{-\varepsilon\delta})^{-1}x d\mu(\lambda)\right\|_X
\\
&&\lesssim \|T(t)B^{-\nu+\varepsilon\delta}\Psi_{\varepsilon\delta}(B^{-1})^{s+1}\|\int_{0+}^{\tau}\frac{\tau}{\lambda+\tau}d\mu(\lambda)\|x\|_X.  
\end{eqnarray}

Since $\beta>\varepsilon\delta$, one has $(s+1)\beta+1/r-\varepsilon\delta> \beta+1/r$; then, it follows from Corollary~\ref{RFINAL}-(i) (by taking $\rho=s-\varepsilon\delta/\beta$) that there exists $C_{\varepsilon\delta,s}>0$ such that for each $t\ge 1$, 
\begin{equation}\label{eqq338}
   \|T(t)B^{-\nu+\varepsilon\delta}\Psi_{\varepsilon\delta}(B^{-1})^{s+1}\|_{\mathcal{L}(X)}\leq C_{\varepsilon\delta,s} t^{-s+\varepsilon\delta/\beta}.
\end{equation}
Thus, by relations~\eqref{eqq337} and~\eqref{eqq338}, one gets
\begin{equation}\label{eqq39}
  \left\|T(t)B^{-\nu+\varepsilon\delta}\Psi_{\varepsilon\delta}(B^{-1})^{s+1}\int_{0+}^{\tau}B^{-\varepsilon\delta}(\lambda+B^{-\varepsilon\delta})^{-1}x d\mu(\lambda)\right\|_X
\lesssim  t^{-s+\varepsilon\delta/\beta} \int_{0+}^{\tau} \frac{\tau}{\lambda+\tau} d\mu(\lambda)\|x\|_X.
\end{equation}
Note also that since $B^{-\varepsilon\delta}$ is a sectorial operator, it follows that for each $t\geq 0$ and each $x \in X$,
\begin{eqnarray}\label{eqq310}
  \nonumber  &&\left\|T(t)B^{-\nu+\varepsilon\delta}\Psi_{\varepsilon\delta}(B^{-1})^{s+1}\int_{\tau}^{\infty}B^{-\varepsilon\delta}(\lambda+B^{-\varepsilon\delta})^{-1}x d\mu(\lambda)\right\|_X\\
\nonumber & \lesssim& \|T(t)B^{-\nu}\Psi_{\varepsilon\delta}(B^{-1})^{s+1}\|\int_{\tau}^{\infty} 
\frac{\|x\|_X}{\lambda} d\mu(\lambda)\\
&\lesssim& 2 \|T(t)B^{-\nu}\Psi_{\varepsilon\delta}(B^{-1})^{s+1}\|\int_{\tau}^{\infty} 
\frac{1}{\lambda+\tau} d\mu(\lambda)\|x\|_X
\end{eqnarray}
By combining relations~\eqref{ig},~\eqref{tau},~\eqref{eqq39} and \eqref{eqq310}, it follows that for each $t\ge 1$, 
\begin{eqnarray*}
\|T(t)B^{-\nu}\Psi_{\varepsilon\delta}(B^{-1})^{m}\|_{\mathcal{L}(X)}&\lesssim& t^{-s+\varepsilon\delta/\beta}\left(\int_{0+}^{\tau}\frac{\tau}{\tau+\lambda} d\mu(\lambda) +\tau\int_{\tau}^{\infty} \frac{1}{\lambda+\tau} d\mu(\lambda)\right)\\
 \nonumber&\lesssim& t^{-s+\varepsilon\delta/\beta} \left(S_{h_{\varepsilon\delta}}(1/
 \tau)\right)^{\delta}. 
\end{eqnarray*}

Now, set $k_{\varepsilon\delta,s}(t):=\|T(t)B^{-\nu}\Psi_{\varepsilon\delta}(B^{-1})^{s+1}\|_{\mathcal{L}(X)}$; then, it follows from Corollary \ref{RFINAL}-(i) that there exists $c_{\varepsilon\delta,s}>0$ such that $k_{\varepsilon\delta,s}(t)\leq c_{\varepsilon\delta,s} t^{-s}$ for each $t\geq 1$. So, for each $t\ge 1$,  
\begin{equation}\label{eqq311}
\dfrac{1}{\tau}=\dfrac{t^{-s+\frac{\varepsilon\delta}{\beta}}}{k_{\varepsilon\delta,s}(t)}\ge \frac{1}{c_{\varepsilon\delta,s}} t^{\frac{\varepsilon\delta}{\beta}}.
\end{equation}Therefore, by Theorem~\ref{lema21} (with $\sigma=1$) and by relation~\eqref{tau}, one has for each sufficiently large $t$ that
\begin{eqnarray}\label{eqq312}
 \nonumber\|T(t)B^{-\nu}\Psi_{\varepsilon\delta}(B^{-1})^{m}\|_{\mathcal{L}(X)}&\lesssim& t^{-s+\varepsilon\delta/\beta} (S_{h_{\varepsilon\delta}}(1/
 \tau))^{\delta}\\
 &\lesssim& k_{\varepsilon\delta,s}(t)\left(h_{\varepsilon\delta}\left(\dfrac{t^{-s+\frac{\varepsilon\delta}{\beta}}}{k_{\varepsilon\delta,s}(t)}\right)\right)^{\delta}.
\end{eqnarray}
It follows from relation~\eqref{eqq311} and Proposition~\ref{cor2}(a) that 
for each sufficiently large $t$, 
\begin{equation}
\label{eqq313}
h_{\varepsilon\delta}\left(\frac{t^{-s+\frac{\varepsilon\delta}{\beta}}}{k_{\varepsilon\delta,s}(t)}\right)\lesssim  \frac{2c_{\varepsilon\delta,s}t^{-s+\frac{\varepsilon\delta}{\beta}}}{t^{\frac{\varepsilon\delta}{\beta}}k_{\varepsilon\delta,s}(t)}(h_{\varepsilon\delta}(t^{\frac{\varepsilon\delta}{\beta}}/(2c_{\varepsilon\delta,s})))^{\delta}\lesssim
\frac{1}{k_{\varepsilon\delta,s}(t)t^{s}}(\ell(t^{\frac{1}{\beta}}))^{\delta}.
\end{equation}

One concludes from~\eqref{eqq312} and~\eqref{eqq313} that for each sufficiently large $t$,
\begin{eqnarray}\label{eqq34}
\nonumber\|T(t)B^{-\nu}\Psi_{\varepsilon\delta}(B^{-1})^{m}\|_{\mathcal{L}(X)}\lesssim 
t^{-s}  \left(\ell(t^{\frac{1}{\beta}})\right)^{\delta}.
\end{eqnarray}

\noindent \textbf{Step 2: removing $m$}. It follows
 from the discussion presented in the beginning of \textbf{Step 1} and from the functional calculus defined by~\eqref{cbf} (note that $S_{h_\varepsilon}(1/\lambda^{\varepsilon})\in \mathcal{CBF}$ and that $\Psi_\varepsilon(B^{-1})S_{h_\varepsilon}(B^{-\varepsilon})=B^{-\varepsilon}$) that for each $x\in X$ and each $t\ge 1$, 
\begin{eqnarray*}
\nonumber T(t)B^{-\nu+\varepsilon}B^{-\varepsilon}\Psi_{\varepsilon}(B^{-1})^{m-1}x&=&T(t)B^{-\nu+\varepsilon}S_{h_\varepsilon}(B^{\varepsilon})\Psi_{\varepsilon}(B^{-1})^{m}x\\ &=&\int_{0}^{\infty}B^{-\varepsilon}(\lambda+B^{-\varepsilon})^{-1}T(t)B^{-\nu+\varepsilon}\Psi_{\varepsilon}(B^{-1})^{m}x d\mu(\lambda).
\end{eqnarray*}

Now one may follow the same steps presented in the proof of \cite[Theorem~1.13]{nos}, with
\[\tau:= \dfrac{\|T(t)B^{-\nu}\Psi_{\varepsilon}(B^{-1})^{s+1}\|_{\mathcal{L}(X)}}{t^{-s+\varepsilon/\beta}}>0,\]
and obtain, for each $t\geq 1$, the relation
\begin{eqnarray*}
\nonumber\|T(t)B^{-\nu}\Psi_{\varepsilon}(B^{-1})^{m-1}\|_{\mathcal{L}(X)}\lesssim 
t^{-s}  \left(\ell(t^{\frac{1}{\beta}})\right)^{\delta+1}.
\end{eqnarray*}

By proceeding recursively over $m$, it follows from the previous discussion that for each $t\ge 1$,
\begin{eqnarray*}
\nonumber\|T(t)B^{-\nu}\|_{\mathcal{L}(X)}\lesssim 
t^{-s}  \left(\ell(t^{\frac{1}{\beta}})\right)^{\delta+m}.
\end{eqnarray*}

\noindent \textbf{Case $\delta=0$.} Since in this case $m\in \mathbb{N}$, one just needs to proceed as in \textbf{Step 2} of the
 case $\delta>0$ in order to obtain, for each $t\geq 1$,
\begin{eqnarray*}
\nonumber\|T(t)B^{-\nu}\|_{\mathcal{L}(X)}\lesssim 
t^{-s}  \left(\ell(t^{\frac{1}{\beta}})\right)^{m}.
\end{eqnarray*}

\

\noindent \textbf{(ii)} The result is equivalent to the following
statement: for each $s\geq 0$, there exist $C_{s}>0$ and $t_0\geq 1$ such that for each $t\geq t_0$,
\begin{equation*}
\|T(t)(1+A)^{-\nu}\|_{\mathcal{L}(X)}\le C_{s} t^{-s}
(1/\kappa(t)^{s+1}),
\end{equation*}
where $\nu:=\beta(s+1)+1/r$   for $p\neq 2$ and $\nu:=\beta(s+1)$ otherwise. Set $m:=\lfloor s+1 \rfloor$ and $\delta:=s+1-\lfloor s+1 \rfloor\in [0,1)$.

\noindent\textbf{Case $\delta>0$.}

\noindent\textbf{Step 1: removing $\delta>0$}. Let $\varepsilon:=\frac{\min\{1,\beta\}}{2}>0$. For each $x\in X$ and each $t\ge 1$, one has (by relation~\eqref{cbf}; note that $S_\kappa(1/\lambda)^{\delta}\in\mathcal{CBF}$)
\begin{eqnarray*}
T(t)B^{-\nu} \Phi(B)^{m}x&=&T(t)B^{-\nu} B^{-1}(S_{\kappa}(B))^{-\delta}\Phi(B)^{s+1}x\\ &=&\int_{0}^{\infty}B^{-1}(\lambda+B^{-1})^{-1}T(t)B^{-\nu}\Phi(B)^{s+1}x d\mu(\lambda).
\end{eqnarray*}

Let 
\begin{equation}\label{vartheta}
    \sigma:= \dfrac{\|T(t)B^{-\nu-1}\Phi(B)^{s+1}\|_{\mathcal{L}(X)}}{t^{-s}}>0;
\end{equation}
by proceeding as in the proof of relation~\eqref{eqq338}, it follows from Corollary~\ref{RFINAL}-(ii) that for each $t\geq 1$,
\begin{equation}
\label{eqq314}
\left\|T(t)B^{-\nu}\Phi(B)^{s+1}\int_{0+}^{\sigma}B^{-1}(\lambda+B^{-1})^{-1}x d\mu(\lambda)\right\|_X\lesssim t^{-s} \int_{0+}^{\sigma} \frac{\sigma}{\lambda+\sigma} d\mu(\lambda)\|x\|_X. 
\end{equation}

Note that for each $t \geq 1$,
\begin{equation}\label{eqq315}
\left\|T(t)B^{-\nu}\Phi(B)^{s+1}\int_{\sigma}^{\infty}B^{-1}(\lambda+B^{-1})^{-1}x d\mu(\lambda)\right\|_X\lesssim\|T(t)B^{-\nu-1}\Phi(B)^{s+1}\|\int_{\sigma}^{\infty} \frac{1}{\lambda+\sigma} d\mu(\lambda)\|x\|_X;
\end{equation}
hence, it follows from relations~\eqref{vartheta}, \eqref{eqq314} and~\eqref{eqq315}  that for each $t\ge 1$,
\begin{eqnarray}\label{eqq29}
  \nonumber\|T(t)B^{-\nu}\Phi(B)^{m}\|_{\mathcal{L}(X)}&\lesssim& t^{-s}\left(\int_{0+}^{\sigma}\frac{\sigma}{\sigma+\lambda} d\mu(\lambda) +\sigma\int_{\sigma}^{\infty} \frac{1}{\lambda+\sigma} d\mu(\lambda)\right)\\
&\lesssim& t^{-s} \frac{\sigma}{S_{\kappa}(1/
 \sigma)^{\delta}}. 
\end{eqnarray}
By combining relation \eqref{eqq29} with Theorem~\ref{lema21}, one gets, for each sufficiently large $t$,
\begin{eqnarray} \label{eq123}
  \|T(t)B^{-\nu}\Phi(B)^{m}\|_{\mathcal{L}(X)} \lesssim  t^{-s} \frac{\sigma}{S_{\kappa}(1/
 \sigma)^{\delta}} \le  \frac{t^{-s}}{(\kappa\left(\frac{t^{-s}}{\|T(t)B^{-\nu-1}\Phi(B)^{s+1}\|_{\mathcal{L}(X)}}\right))^{\delta}}. 
\end{eqnarray}
Now, it follows from Corollary~\ref{RFINAL}-(ii) that for each $t \ge 1$, 
\begin{equation}\label{eqq317}
\|T(t)B^{-\nu-1} \Phi(B)^{s+1}\|_{\mathcal{L}(X)}\lesssim \|T(t)B^{-\nu-1}\Phi(B)^{s+1+1/\beta}\|_{\mathcal{L}(X)}\lesssim t^{-s-\frac{1}{\beta}}
\end{equation}
(note that $\Phi(B)^{-1/\beta}\in \mathcal{L}(X)$), and so
%
\begin{equation*}
  \frac{t^{-s}}{\|T(t)B^{-\nu-1}\Phi(B)^{s+1}\|_{\mathcal{L}(X)}} \gtrsim t^{\frac{1}{\beta}}.
\end{equation*}
Since $\kappa$ is an increasing function, it follows from relation~\eqref{eq123} that 
\begin{equation*}
\|T(t)B^{-\nu}\Phi(B)^{m}\|_{\mathcal{L}(X)} \lesssim \frac{t^{-s}}{(\kappa(t^{1/\beta}))^{\delta}}.
\end{equation*}

\noindent \textbf{Step 2: removing $m$}. By using arguments that are similar to those presented in item \textbf{(i)}, one has for each $t\geq 1$,
\begin{equation*}
\|T(t)B^{-\nu}\|_{\mathcal{L}(X)} \lesssim \frac{t^{-s}}{\kappa(t^{1/\beta})^{m+\delta}}.
\end{equation*}

\noindent \textbf{Case $\delta=0$.} Since in this case $m\in \mathbb{N}$, one just needs to proceed as in \textbf{Step 2} of the
 case $\delta>0$ in order to obtain, for each $t\geq 1$,
 \begin{equation*}
\|T(t)B^{-\nu}\|_{\mathcal{L}(X)} \lesssim \frac{t^{-s}}{\kappa(t^{1/\beta})^{m}}.
\end{equation*}
\end{proof}

\begin{center} 
\Large{Acknowledgments} 
\end{center}
We thank the anonymous reviewers for their valuable contributions, which were essential for improving this manuscript.GSS thanks the partial support by UESB. 


\noindent{\bf{Data availability statement}} Data sharing not applicable to this article as no datasets were generated or analyzed during the current study.

%

\noindent{\bf{Declarations}}


\noindent{\bf{Conflict of interest}}  On behalf of all authors, the corresponding author states that there is no conflict of interest.


\noindent  Email: genilson.santana@uesb.edu.br, Departamento de Ciências Exatas e Tecnológicas, UESB, Vitória da Conquista, BA, 45083-900 Brazil.

\noindent  Email: silas@mat.ufmg.br, Departamento de Matem\'atica, UFMG, Belo Horizonte, MG, 30161-970 Brazil.

\end{document}